\documentclass[a4paper,12pt]{article}
\usepackage{amsmath}
\usepackage{amsthm}
\usepackage{graphicx}
\usepackage{verbatim}
\usepackage{epsfig}
\usepackage{hyperref}
\usepackage{float}
\usepackage{color}
\usepackage{fullpage}
\usepackage{amsfonts}
\usepackage{amscd}
\usepackage{amssymb}
\usepackage{graphics}
\usepackage{amsmath}
\usepackage[english]{babel}
\usepackage{bbm}
\usepackage{yfonts}
\usepackage{mathrsfs}
\usepackage{color}
\DeclareFontFamily{OML}{rsfs}{\skewchar\font'177}
\DeclareFontShape{OML}{rsfs}{m}{n}{ <5> <6> rsfs5 <7> <8> <9>
rsfs7 <10> <10.95> <12> <14.4> <17.28> <20.74> <24.88> rsfs10 }{}
\DeclareMathAlphabet{\mathfs}{OML}{rsfs}{m}{n}
\newcommand{\BE}{{\mathbb{E}}}

\newcommand{\BP}{{\mathbb{P}}}

\newcommand{\BR}{{\mathbb{R}}}

\newcommand{\BZ}{{\mathbb{Z}}}

\newcommand{\CA}{{\mathcal{A}}}

\newcommand{\CC}{{\mathcal{C}}}
\renewcommand{\CD}{{\mathcal{D}}}
\newcommand{\CE}{{\mathcal{E}}}

\newcommand{\CL}{{\mathcal{L}}}

\newcommand{\CP}{{\mathcal{P}}}

\newcommand{\CS}{{\mathcal{S}}}

\newcommand{\bae}{\begin{equation}\begin{aligned}}
\newcommand{\eae}{\end{aligned}\end{equation}}

\newcommand{\cc}{{\mathfrak c}}
\newcommand{\diam}{{\rm diam}}
\newtheorem{thm}{Theorem}[section]
\newtheorem{prop}[thm]{Proposition}
\newtheorem{lem}[thm]{Lemma}

\newtheorem{cor}[thm]{Corollary}

\begin{document}
\numberwithin{equation}{section} \numberwithin{figure}{section}
\title{Connectedness of Poisson cylinders in Euclidean space}
\author{Erik I. Broman\footnote{Department of Mathematics, Uppsala University, Sweden. E-mail: broman@math.uu.se} 
and Johan Tykesson\footnote{Department of Mathematics, Uppsala University, Sweden. E-mail: johan.tykesson@gmail.com}}
\maketitle
\begin{abstract}
We consider the Poisson cylinder model in ${\mathbb R}^d$, $d\ge 3$. We show that given any two
cylinders ${\mathfrak c}_1$ and ${\mathfrak c}_2$ in the process, there is a sequence of at most
$d-2$ other cylinders creating a connection between ${\mathfrak c}_1$ and ${\mathfrak c}_2$. In
particular, this shows that the union of the cylinders is a connected set, answering a question
appearing in \cite{TW}. We also show that there are cylinders in the process that are not connected
by a sequence of at most $d-3$ other cylinders. Thus, the diameter of the cluster of cylinders equals $d-2$.
\end{abstract}
%\tableofcontents
%*****************************************************************************************************************************************************
%***************************************************** Section I - Introduction **********************************************************************
%*****************************************************************************************************************************************************
\section{Introduction}\label{s.intro}

This paper is devoted to the study of the geometry of a random collection of bi-infinite cylinders
in ${\mathbb R}^d$, $d\ge 3$. Before we give the precise definition of this model in
Section~\ref{s.notation}, we describe it informally.

We start with a homogenous Poisson line process $\omega$ of intensity $u\in(0,\infty)$ in ${\mathbb R}^d$.
As the parameter $u$ will play a very little role in this paper, we will
denote its associated probability measure by ${\mathbb P}$ and keep the dependence on $u$ implicit.
Around each line $L\in \omega$, we then center a bi-infinite cylinder $\cc(L)$ of base-radius $1$.
We will sometimes abuse notation and say that $\cc(L)\in \omega.$
The union over $\omega$ of all
cylinders is a random subset of ${\mathbb R}^d$ and we call it ${\mathcal C}$. We think of
${\mathcal C}$ as the \emph{covered region} and its complement
${\mathcal V}:={\mathbb R}^d\setminus {\mathcal C}$ as the \emph{vacant region}. 
We will refer to this model as the Poisson cylinder model, and before we move on to describe our results,
we will discuss some previous results. The model was first suggested by I. Benjamini to the second author
\cite{IB} and subsequently studied in \cite{TW}. In~\cite{TW}, the focus was on the existence of a
non-degenerate percolative phase transition in ${\mathcal V}$ (see \cite{M} for a general text
on continuum percolation models). Indeed, letting
\[
u_*(d):=\sup\{u\,:\,{\mathcal V}\mbox{ has infinite connected components a.s.}\},
\]
it was proved that $0<u_*(d)<\infty$ for every $d\geq 4,$ and that $u_*(3)<\infty.$
Later, it was proved in~\cite{HST} that $u_*(3)>0.$

In the present paper, we focus on connectivity properties of ${\mathcal C}$. To that end,
for any $\cc_a,\cc_b\in \omega$ we let the {\em cylinder distance} ${\rm Cdist}(\cc_a,\cc_b)$
be the minimal number $k$ such that there exists cylinders $\cc_1,\ldots\cc_k\in \omega$ so that
\[
\cc_a\cup\cc_b \bigcup_{i=1}^k\cc_k,
\]
is a connected set. We then define the {\em diameter} of $\CC$ as
\[
\diam(\CC)=\sup\{{\rm Cdist}(\cc_a,\cc_b):\cc_a,\cc_b\in \omega\}.
\]
Our main result is as follows.
\begin{thm}\label{t.maind}
For any $d\geq 3,$
\[
{\mathbb P}[\diam(\CC)=d-2]=1.
\]
\end{thm}
\noindent
{\bf Remarks:} 
The proof in the case $d=3$ is very different from the proof in $d\geq 4,$ and therefore they
will be presented in different sections.

When $d=2,$ every line in a Poisson line process a.s.\
intersects every other line in the process, so that trivially $\BP[\diam(\CC)=0]=1.$

It is an easy consequence of scaling, that Theorem \ref{t.maind} holds also for cylinders with
radius different from 1. Considering a model with random radii, it will still be the case that a.s.
$\diam(\CC)\leq d-2$ (unless the distribution of the radii are degenerate). This follows from an easy
coupling argument.

It is interesting to note that if we define
\[
u_c(d):=\inf\{u: \exists \textrm{ a unique component of } \CC(u,\omega)
\textrm{ containing infinitely many cylinders a.s.} \}
\]
(which is very natural in the context of percolation models), then Theorem \ref{t.maind}
implies that $u_c(d)=0$ for every $d\geq 3.$ In fact, Theorem \ref{t.maind} tell us that the union of the 
cylinders consists solely of a unique, infinite, component.
This is in sharp contrast to similar results for other continuum percolation models 
as well as for discrete percolation models
(see for example \cite{M} and \cite{Grimmett}). In those settings, the phase transition 
is non-trivial in that the critical parameter value is strictly bounded away from 0.
However, proofs
of such results usually rely on some sort of local dependencies and/or so-called finite energy conditions.
Our case is quite different, since our model lack these features. Indeed,
the long range dependence in ${\mathcal C}$ and ${\mathcal V}$ manifests itself
in for example the following way (see \cite{TW} Equation (3.9)):
\begin{equation}\label{e.corrdec}
\frac{c(d,u)}{|x-y|^{d-1}}\le \text{cov}_u({\bf 1}\{x\in {\mathcal V}\},{\bf 1}\{y\in {\mathcal V}\})
\le \frac{c'(d,u)}{|x-y|^{d-1}}
\end{equation}
as soon as $|x-y|>2$, and for some constants $c(d,u),c'(d,u)\in(0,\infty)$ independent of $x,y$.
This long range dependence creates challenges in the study of ${\mathcal C}$ and ${\mathcal V}$
as techniques developed for percolation models exhibiting only bounded range dependence are often
not applicable.
In fact, the lack of the mentioned features is one of the main motivations
for studying the model. 

\bigskip

The study of these questions is partly inspired by some recent works on the
\emph{random interlacements process} on ${\mathbb Z}^d$, $d\ge 3$, introduced
%by A.-S.\ Sznitman
in \cite{S1}. The random interlacement is a discrete percolation model obtained by a
Poissonian collection of bi-infinite random walk trajectories. It was shown in
\cite{RS} and \cite{PT} that given any two trajectories in the random interlacement, there
is some sequence of at most $\lceil d/2\rceil-2$ other trajectories connecting them.

In \cite{PT}, methods based on the concept of stochastic dimension from \cite{BKPS} were successfully used.
However, it turns out that those methods cannot be applied for the Poisson cylinder model. This is mainly
because the long range dependence in the Poisson cylinder model 
is of a different nature than the long range dependence in the random
interlacements model. For random interlacements, inequalities similar to~\eqref{e.corrdec} hold, but
with $d-1$ replaced by $d-2$.

In order to show our results we thus had to take other routes. The proof in the case $d=3$ relies on a projection
method combined with an integral formula from \cite{SW} to show that the number of lines intersecting any two
cylinders $\cc_a,\cc_b$ is a.s. infinite. When $d\geq 4,$ then to prove the lower bound of Theorem~\ref{t.maind},
we adapted a method from \cite{RS}. The proof of the upper bound of Theorem~\ref{t.maind} is a rather involved
use of the second moment method. In fact, the proof of this upper bound occupies more than half of the paper.

The rest of the paper is organized as follows. In Section~\ref{s.notation} we define the
Poisson cylinder model precisely. In Section~\ref{s.ri3d} we give the proof of Theorem~\ref{t.maind} for $d=3$.
Some preliminary measure estimates needed for the proof of Theorem~\ref{t.maind} when $d\geq 4,$
are given in Section~\ref{s.prelim}. Finally, the proofs of the lower and upper bounds of Theorem~\ref{t.maind}
are given in Section~\ref{s.lbound} and Section~\ref{s.maind} respectively.

\section{Notation and definitions}\label{s.notation}

We let $A(d,1)$ be the set of bi-infinite lines in ${\mathbb R}^d$. Let $G(d,1)$ be the set of
bi-infinite lines in ${\mathbb R}^d$ that pass through the origin. In other words, $A(d,1)$ is
the set of $1$-dimensional affine subspaces of ${\mathbb R}^d$, and $G(d,1)$ is the set of
$1$-dimensional linear subspaces of ${\mathbb R}^d$. Subsets of $G(d,1)$ and $A(d,1)$ will
typically be denoted by scripted letters like $\CA$ and $\CL$. If $K\in {\mathcal B}({\mathbb R}^d)$
we let $\CL_K\subset A(d,1)$ denote the set of lines that intersect $K$:
\[
\CL_K=\{L\in {A(d,1)}\,:\,L\cap K\neq \emptyset\}.
\]

Let $B_d(0,1)$ denote the $d$-dimensional ball of radius $1$ and let $\kappa_d$ denote the volume of $B_d(0,1)$.
On $G(d,1)$ there is a unique Haar measure $\nu_{d,1}$, normalized so that $\nu_{d,1}(G(d,1))=1$, and on $A(d,1)$,
there is a unique Haar measure $\mu_{d,1}$ normalized so that $\mu_{d,1}(\CL_{B^d(0,1)})=\kappa_{d-1}$
(see for instance \cite{SW} Chapter 13).
We let $SO_d$ be the rotation group on ${\mathbb R}^d$. Typically, we think of the elements of
$SO_d$ as the orthogonal
$d\times d$ matrices with determinant $1$. For any subspace
$H\subset \BR^d,$ and set $A\subset \BR^d,$ we let $\Pi_{H}(A)$
denote the projection of $A$ onto $H.$ In these cases, we will consider $\Pi_{H}(A)$ as a subset of $\BR^d.$
We will let $e_1,e_2,\ldots,e_d$
denote the generic orthonormal set of vectors that span $\BR^d.$
%Furthermore, we will let $L^i:=\{te_i:-\infty<t<\infty\}.$

\subsection{The Poisson cylinder model}
We consider the
following space of point measures on $A(d,1)$:
%\begin{equation}\label{d.omegadef}
\[
\Omega=\{\omega=\sum_{i=0}^{\infty}\delta_{L_i}\text{ where $L_i\in A(d,1)$, and }
\omega(\CL_A)<\infty\text{ for all compact }A\subset {\mathbb R}^d\}.
\]
%\end{equation}
Here, $\delta_L$ of course denotes point measure at $L.$

In what follows, we will often use the following standard abuse of notation: if $\omega$ is some point
measure, the expression $"x\in \omega"$ will stand for $"x\in \text{supp}(\omega)"$. If
$\omega\in \Omega$ and $A\in {\mathcal B}({\mathbb R}^d)$ we let $\omega_A$ denote the
restriction of $\omega$ to $\CL_A$.
We will draw an element $\omega$ from $\Omega$ according to a Poisson point process with intensity
measure $u\mu_{3,1}$ where $u>0$.  We call $\omega$ a \emph{(homogeneous) Poisson line process} of
intensity $u$.%, and denote the law corresponding to it by ${\mathbb P}_u$.

If $L\in A(d,1)$, we denote by ${\mathfrak c}(L)$ the cylinder of base radius $1$ centered around $L$:
%\begin{equation}
\[
{\mathfrak c}(L)=\{x\in {\mathbb R}^d\,:\,d(x,L)\le 1\}.
\]
%\end{equation}
Finally the object of main interest in this paper, the union of all cylinders is denoted by ${\mathcal C}$:
$${\mathcal C}={\mathcal C}(\omega)=\bigcup_{L\in \omega} {\mathfrak c}(L).$$

\section{Proof of Theorem~\ref{t.maind} when $d=3$}\label{s.ri3d}

The aim of this section is to prove the following theorem.
\begin{thm} \label{thm:maind=3}
For $d=3,$ 
\[
\BP(\diam(\CC)=1)=1.
\]
\end{thm}

In Section~\ref{s.2cyl3d} we consider two arbitrary fixed cylinders $\cc_1,\cc_2$ and show that 
the $\mu_{3,1}$-measure of the set of lines that intersect both of them 
is infinite, see Proposition~\ref{prop:cylinder3d}. It will then be straightforward to 
prove Theorem \ref{thm:maind=3}, which we do in Section~\ref{s.3dfinish}.

\subsection{Lines intersecting two cylinders in three dimensions}\label{s.2cyl3d}
We write
$$L=\{t(l_1,l_2,l_3)\,:\,-\infty<t<\infty\}$$ for a line in $G(3,1),$ where $l_1^2+l_2^2+l_3^2=1$.
\begin{prop}\label{prop:cylinder3d}
For any two lines $L_1,L_2 \in A(3,1)$,
\[
\mu_{3,1}({\mathcal L}_{\mathfrak{c}(L_1)}\cap {\mathcal L}_{\mathfrak{c}(L_2)})=\infty.
\]
\end{prop}

\begin{proof}
We will first consider the case when $L_1,L_2\in G(3,1)$, and then
explain how to obtain the general case. By invariance of $\mu_{3,1}$ under translations and
rotations of ${\mathbb R}^3$, we can without loss of generality assume that $L_1=\{te_1:-\infty<t<\infty\}$.
On the other hand,
\[
L_2:=\{t(k_1,k_2,k_3)\,:\,-\infty<t<\infty\},
\]
for some $k_1^2+k_2^2+k_3^2=1.$
By the representation of \cite{SW} Theorem $13.2.12$ we have
\begin{equation}\label{e.measrepr}
\begin{split}
\mu_{3,1}(&{\mathcal L}_{\mathfrak{c}(L_1)}\cap {\mathcal L}_{\mathfrak{c}(L_2)})\\ &
= \int_{G(d,1)} \int_{L^\bot} {\bf 1}\{L+y\in {\mathcal L}_{\mathfrak{c}(L_1)}\cap
{\mathcal L}_{\mathfrak{c}(L_2)}\} \lambda_2(dy)\nu_{3,1}(dL),
\end{split}
\end{equation}
where $\lambda_2$ denotes two-dimensional Lebesgue measure.
Observe that for fixed $L,$ the set of $y\in L^\bot$ such that
$y+L\in {\mathcal L}_{\mathfrak{c}(L_i)}$ is exactly $\Pi_{L^{\bot}}(\mathfrak{c}(L_i))$
for $i=1,2.$
%$\cap {\mathcal L}_{\mathfrak{c}(L_2)}$
%if and only if $y \in \Pi_{L^{\bot}}(\mathfrak{c}(L_1))\cap \Pi_{L^{\bot}}(\mathfrak{c}(L_2))$.
Hence,
\begin{equation}\label{e.projmethod}
\int_{L^\bot} {\bf 1}\{L+y\in {\mathcal L}_{\mathfrak{c}(L_1)}\cap {\mathcal L}_{\mathfrak{c}(L_2)}\} \lambda_2(dy)
=\lambda_2\left(\Pi_{L^{\bot}}(\mathfrak{c}(L_1))\cap \Pi_{L^{\bot}}(\mathfrak{c}(L_2))\right).
\end{equation}

Let $K=K(L):=\Pi_{L^{\bot}}(\mathfrak{c}(L_1))\cap \Pi_{L^{\bot}}(\mathfrak{c}(L_2)).$
The sets $\Pi_{L^\bot}(\mathfrak{c}(L_1))$ and $\Pi_{L^\bot}(\mathfrak{c}(L_2))$ are two-dimensional
cylinders of width $2$ in $L^\bot$, with central axes $\Pi_{L^\bot}(L_1)$ and $\Pi_{L^\bot}(L_2)$ respectively.
Therefore, $K$ is a rhombus except when $\Pi_{L^\bot}(L_1)=\Pi_{L^\bot}(L_2)$ in which case $K$ is an infinite 
strip. If $s(L)$ denotes the sidelength of $K,$ then since the height of $K$ is $2$, 
we have $\lambda_2(K)=2 s(L)$.

We will now find $s(L)$.  Write
$$\tilde{L}_1:=\Pi_{L^\bot}(L_1)\mbox{ and }\tilde{L}_2:=\Pi_{L^\bot}(L_2).$$ For $i=1,2$, the boundary of
 $\Pi_{L^\bot}(\mathfrak{c}(L_i))$ consists of two lines which we denote by $\tilde{L}_i^{'}$ and 
 $\tilde{L}_i^{''}$. One of the sides of $K$ is given by the linesegment between the points
  $\tilde{L}_2^{'}\cap\tilde{L}_1^{'}$ and $\tilde{L}_2^{'}\cap\tilde{L}_1^{''}$. Hence
$$s(L)=|\tilde{L}_2^{'}\cap\tilde{L}_1^{'}-\tilde{L}_2^{'}\cap\tilde{L}_1^{''}|.$$
However, since $\tilde{L}_2$ and $\tilde{L}_2^{'}$ are parallell, and $\tilde{L}_1^{'}$ 
and $\tilde{L}_1^{''}$ are parallel, we also have
\[
s(L)=|\tilde{L}_2\cap\tilde{L}_1^{'}-\tilde{L}_2\cap\tilde{L}_1^{''}|.
\]
Furthermore, since $\tilde{L}_1$ 
is at distance $1$ from both $\tilde{L}_1^{'}$ and $\tilde{L}_1^{''}$, it follows that
\[
s(L)=2|\tilde{L}_2\cap\tilde{L}_1^{'}-\tilde{L}_2\cap\tilde{L}_1|=2 |\tilde{L}_2\cap\tilde{L}_1^{'}|,
\]
using that $\tilde{L}_2\cap\tilde{L}_1=\{o\}$.

We aim to express $|\tilde{L}_2\cap\tilde{L}_1^{'}|$ in terms of the directional vectors of 
$\tilde{L}_1$ and $\tilde{L}_2$. We write the projection $\Pi_{L^\bot}$ in matrix form:
$$\Pi_{L^\bot}:=
\begin{bmatrix}
l_2^2+l_3^2 & -l_1 l_2 & -l_1 l_3 \\
-l_1 l_2 & l_1^2+l_3^2 & -l_2 l_3 \\
-l_1 l_3 & -l_2 l_3 & l_1^2+l_2^2
\end{bmatrix}.$$
Let
$$v_2:=\Pi_{L^\bot} \left[\begin{matrix}k_1 \\ k_2 \\ k_3\end{matrix}\right]
=\begin{bmatrix}(l_2^2+l_3^2)k_1-l_1 l_2 k_2-l_1 l_3 k_3 \\
-l_1 l_2 k_1+(l_1^2+l_3^2)k_2-l_2 l_3 k_3 \\
 -l_1 l_3 k_1-l_2 l_3 k_2 +(l_1^2+l_2^2)k_3  \end{bmatrix},$$
and
$$v_1:=\Pi_{L^\bot}(e_1)=\begin{bmatrix}l_2^2+l_3^2 \\ -l_1 l_2 \\ -l_1 l_3\end{bmatrix},$$
so that $v_1$ and $v_2$ are directional vectors of $\tilde{L}_1$ and $\tilde{L}_2$ respectively.

For $t\in {\mathbb R}$, let $p_t=t v_2$ denote a point on $\tilde{L}_2$. Then for some $t^*$, 
$p_{t^*}=\tilde{L}_2\cap\tilde{L}_1^{'}$.
Since $\tilde{L}_1^{'}$ and $\tilde{L}_1$ are parallell and at distance $1$, the distance between 
$\tilde{L}_2\cap\tilde{L}_1^{'}$ and $\tilde{L}_1$ is $1$. On the other hand, the distance between any point
$p_t$ and $\tilde{L}_1$ is given by the following elementary formula:% for the distance between a point and a line:
$$d(p_t,\tilde{L}_1)=\frac{|v_1\times p_t|}{|v_1|}=|t| \frac{|v_1\times v_2|}{|v_1|}.$$
The solution $t^*$ to the equation $d(p_t,\tilde{L}_1)=1$ is thus given by
$$t^*=\pm\frac{|v_1|}{|v_1\times v_2|}.$$ Hence, the point $\tilde{L}_2\cap\tilde{L}_1^{'}$ is 
either $(|v_1|/|v_1\times v_2|) v_2$ or $-(|v_1|/|v_1\times v_2|) v_2$ so that
$$|\tilde{L}_2\cap\tilde{L}_1^{'}|=\frac{|v_1||v_2|}{|v_1\times v_2|}.$$
We get that
\begin{equation}\label{e.Aexpression}
\lambda_2(K)=2 s(L)=4 |\tilde{L}_2\cap\tilde{L}_1^{'}| =\frac{4 |v_1| |v_2|}{|v_1\times v_2|}.
\end{equation}
Furthermore,
\begin{equation}\label{e.v1calc}
\begin{split}
|v_1|^2  =(l_2^2+l_3^2)^2+l_1^2 l_2^2+l_1^2 l_3^2 = (l_2^2+l_3^2)(l_1^2+l_2^2+l_3^2) =(l_2^2+l_3^2),
\end{split}
\end{equation}
and
\begin{eqnarray}\label{e.v2calc}
\lefteqn{|v_2|^2 = (k_1 l_2^2+k_1 l_3^2 -l_1 l_2 k_2-l_1 l_3 k_3)^2+(l_1 l_2 k_1-k_2 l_1^2-k_2 l_3^2+l_2 l_3 k_3)^2}\\
& & \ \ +(l_1 l_3 k_1+l_2 l_3 k_2-k_3 l_1^2-k_3 l_2^2)^2\nonumber\\
& & = (l_1^2+l_2^2+l_3^2)((k_3 l_1-k_1 l_3)^2+(k_2 l_1 - k_1 l_2)^2+(k_2 l_3-k_3 l_2)^2)\nonumber\\
& &=(k_3 l_1-k_1 l_3)^2+(k_2 l_1 - k_1 l_2)^2+(k_2 l_3-k_3 l_2)^2,\nonumber
\end{eqnarray}
%\begin{equation}\label{e.v2calc}
%\begin{split}
%|v_2|^2 & =  (k_1 l_2^2+k_1 l_3^2 -l_1 l_2 k_2-l_1 l_3 k_3)^2+(l_1 l_2 k_1-k_2 l_1^2-k_2 l_3^2+l_2 l_3 k_3)^2\\
%& \ \ +(l_1 l_3 k_1+l_2 l_3 k_2-k_3 l_1^2-k_3 l_2^2)^2\\
%& = (l_1^2+l_2^2+l_3^2)((k_3 l_1-k_1 l_3)^2+(k_2 l_1 - k_1 l_2)^2+(k_2 l_3-k_3 l_2)^2)\\
%& =(k_3 l_1-k_1 l_3)^2+(k_2 l_1 - k_1 l_2)^2+(k_2 l_3-k_3 l_2)^2,
%\end{split}
%\end{equation}
where the second equality follows from straightforward but lengthy calculations.
Further calculations give
%\begin{equation}\label{e.v1v2crossprod}
\[
v_1\times v_2
=\begin{bmatrix}l_1 l_2^2 l_3 k_2-l_1^3 l_2 k_3-l_1 l_2^3 k_3+l_1^3 l_3 k_2 + l_1 l_3^3 k_2-l_1 l_3^2 l_2 k_3 \\
l_1^2 l_3 l_2 k_2+l_2^3 l_3 k_2-l_2^2 k_3 l_1^2-l_2^4 k_3+l_3^3 l_2 k_2-l_3^2 k_3 l_2^2\\
l_2^2 k_2 l_3^2-l_2^3 l_3 k_3+l_3^2 k_2 l_1^2+l_3^4 k_2-l_3^3 l_2 k_3-l_1^2 l_2 l_3 k_3
\end{bmatrix},
\]
%\end{equation}
%Finally we calculate $|v_1\times v_2|:$
so that
\begin{eqnarray}\label{e.vectorprodcalc}
\lefteqn{|v_1\times v_2|^2 
=(l_1 l_2^2 l_3 k_2-l_1^3 l_2 k_3-l_1 l_2^3 k_3+l_1^3 l_3 k_2 + l_1 l_3^3 k_2-l_1 l_3^2 l_2 k_3)^2}\\
& & \ \ +(l_1^2 l_3 l_2 k_2+l_2^3 l_3 k_2-l_2^2 k_3 l_1^2-l_2^4 k_3+l_3^3 l_2 k_2-l_3^2 k_3 l_2^2)^2\nonumber\\
& & \ \ +(l_2^2 k_2 l_3^2-l_2^3 l_3 k_3+l_3^2 k_2 l_1^2+l_3^4 k_2-l_3^3 l_2 k_3-l_1^2 l_2 l_3 k_3)^2\nonumber\\
& & =((l_1^2+l_2^2+l_3^2)(l_1 l_3 k_2-l_1 l_2 k_3))^2 +((l_1^2+l_2^2+l_3^2)(l_2 l_3 k_2-l_2^2 k_3))^2\nonumber\\
& &\ \ + ((l_1^2+l_2^2+l_3^2)(k_2 l_3^2-l_2 l_3 k_3))^2\nonumber\\
& &=(l_1 l_3 k_2-l_1 l_2 k_3)^2+(l_2 l_3 k_2-l_2^2 k_3)^2+(k_2 l_3^2-l_2 l_3 k_3)^2\nonumber\\
 & &=(l_1^2+l_2^2+l_3^2)(k_2 l_3-l_2 k_3)^2=(k_2 l_3-l_2 k_3)^2.\nonumber
\end{eqnarray}
Combining~\eqref{e.Aexpression},~\eqref{e.v1calc},~\eqref{e.v2calc} and~\eqref{e.vectorprodcalc} we obtain
%\begin{equation}\label{e.Asimple}
\[
\lambda_2(K)=\frac{4\sqrt{(l_2^2+l_3^2)((k_3 l_2-k_2 l_3)^2+(k_1 l_2-k_2 l_1)^2+(k_3 l_1-k_1 l_3)^2)}}{|k_2 l_3-k_3l_2 |}.
\]
%\end{equation}
%Using the expression~\eqref{e.Asimple}, it is straightforward to see that
Using for instance polar coordinates, it is a straightforward exercise to show that 
\begin{equation} \label{eqn:infarea}
\int_{G(3,1)}\lambda_2(K(L))\nu_{3,1}(dL)=\infty.
\end{equation}
Combining \eqref{e.measrepr}, \eqref{e.projmethod} and \eqref{eqn:infarea}
finishes the proof in the case when $L_1,L_2\in G(3,1)$.

We now address the general case. Therefore, let $L_1,L_2\in A(3,1)$ and write
\[
L_1=L_1'+u_1\mbox{ and }L_2=L_2'+u_2,
\]
where $L_1',L_2'\in G(3,1)$ and $u_1,u_2\in {\mathbb R}^3$. The set
\begin{equation} \label{eqn:set}
\Pi_{L^{\bot}}(\mathfrak{c}(L_1))\cap \Pi_{L^{\bot}}(\mathfrak{c}(L_2))
\end{equation}
is still a rhombus,
created by the intersection of two $2$-dimensional cylinders of width $2$ in $L^\bot$.
By an elementary formula for the area of a rhombus, \eqref{eqn:set} equals $4/\sin(\theta)$, 
where $\theta$ is the angle between the
lines $\Pi_{L^\bot}(L_1)$ and $\Pi_{L^\bot}(L_2)$. Observe that
$\Pi_{L^\bot}(L_i)=\Pi_{L^\bot}(L_i')+\Pi_{L^\bot}(u_i)$ for $i=1,2$.
Hence, the angle between $\Pi_{L^\bot}(L_1)$ and $\Pi_{L^\bot}(L_2)$ is the same as the
angle between $\Pi_{L^\bot}(L_1')$ and $\Pi_{L^\bot}(L_2')$. It follows that
\begin{equation}\label{e.rhombareas}
\lambda_2\left(\Pi_{L^{\bot}}(\mathfrak{c}(L_1))\cap \Pi_{L^{\bot}}(\mathfrak{c}(L_2))\right)
=\lambda_2\left(\Pi_{L^{\bot}}(\mathfrak{c}(L_1'))\cap \Pi_{L^{\bot}}(\mathfrak{c}(L_2'))\right).
\end{equation}
From~\eqref{e.measrepr}, ~\eqref{e.projmethod} and~\eqref{e.rhombareas} we get
\[
\mu_{3,1}({\mathcal L}_{\mathfrak{c}(L_1)}\cap {\mathcal L}_{\mathfrak{c}(L_2)})
=\mu_{3,1}({\mathcal L}_{\mathfrak{c}(L_1')}\cap {\mathcal L}_{\mathfrak{c}(L_2')})=\infty,
\]
as above.
\end{proof}

From Proposition~\eqref{prop:cylinder3d}, the following corollary is easy.
\begin{cor}\label{c.connectivity2}
Let $d=3$. Fix $L_1,L_2\in A(3,1)$. For any $u>0$, we have
%\begin{equation}\label{e.connectivity2}
\[
{\mathbb P}[\omega(\CL_{{\mathfrak c}(L_1)}\cap \CL_{{\mathfrak c}(L_2)})=\infty]=1.
\]
%\end{equation}
\end{cor}
\begin{proof}
Follows trivially from Proposition~\eqref{prop:cylinder3d}.
\end{proof}

\subsection{Proof of Theorem~\ref{thm:maind=3}}\label{s.3dfinish}
\emph{Proof of Theorem~\ref{thm:maind=3}.}
For lines $L_1,L_2\in A(3,1)$, let
\[
E(L_1,L_2)=\{\omega(\CL_{{\mathfrak c}(L_1)}\cap \CL_{{\mathfrak c}(L_2)})=\infty\}.
\]
%Define $\tilde{E}(L_1,L_2)$ in the same way as $E(L_1,L_2)$, but with $\omega$ replaced by
%$\omega+\delta_{L_1}+\delta_{L_2}$. 
We know from Corollary~\ref{c.connectivity2} that
\begin{equation}\label{e.eprob}
{\mathbb P}[E(L_1,L_2)]=1\mbox{ for all }L_1,L_2\in A(3,1).
\end{equation}
Let $D:=\cap_{(L_1,L_2)\in \omega^2_{\neq}}E(L_1,L_2).$
Here $\omega^2_{\neq}$ denotes the set of all $2$-tuples of distinct lines from $\omega$.
Observe that if $D$ occurs, then ${\mathcal C}$ is connected, and moreover any two
cylinders are connected via some other cylinder. Hence it suffices to show that
${\mathbb P}[D]=1$. This is intuitively clear in view of~\eqref{e.eprob}, but we now
make this precise. Observe that
\[
D=\left\{\sum_{(L_1,L_2)\in \omega^2_{\neq}} I(E(L_1,L_2)^c)=0\right\},
\]
so that it suffices to show 
\[
{\mathbb E}\left[\sum_{(L_1,L_2)\in \omega^2_{\neq}} I(E(L_1,L_2)^c)\right]=0.
\]
%\end{equation}
Let $\BE^{L_1,L_2}$ denote expectation with respect to $\omega+\delta_{L_1}+\delta_{L_2}.$
According to the Slivnyak-Mecke formula (see \cite{SW} Corollary $3.2.3$) we have
\begin{eqnarray*}
\lefteqn{{\mathbb E}\left[\sum_{(L_1,L_2)\in \omega^2_{\neq}} I(E(L_1,L_2)^c)\right]}\\
& & =\int_{A(3,1)} \int_{A(3,1)}{\mathbb E}^{L_1,L_2}\left[I(E(L_1,L_2)^c)\right] \nu_{3,1}(d L_1) \nu_{3,1}(d L_2) \\
& & = \int_{A(3,1)} \int_{A(3,1)}{\mathbb P}^{L_1,L_2}\left[E(L_1,L_2)^c\right] \nu_{3,1}(d L_1) \nu_{3,1}(d L_2) \\
& & = \int_{A(3,1)} \int_{A(3,1)}{\mathbb P}\left[E(L_1,L_2)^c\right] \nu_{3,1}(d L_1) \nu_{3,1}(d L_2) \\
& & = \int_{A(3,1)} \int_{A(3,1)} 0\, \nu_{3,1}(d L_1) \nu_{3,1}(d L_2)= 0,
\end{eqnarray*}
where the penultimate equality follows from Corollary~\ref{c.connectivity2}. This proves that 
$\BP[\diam(\CC)\leq 1]=1.$

It is an immediate consequence from the Poissonian nature of the model, that with probability 1 
there exists two cylinders $\cc_1,\cc_2\in \omega$ such that $\cc_1\cap \cc_2\neq \emptyset.$
Therefore, $\BP[\diam(\CC)>0]=1.$
\fbox{}\\

\section{Preliminary results in $d$ dimensions}\label{s.prelim}

\subsection{Measure of the set of lines that intersect two balls}

The following proposition is a stronger version of Lemma $3.1$ from \cite{TW}. We give a fundamentally
different proof based on the method of Section \ref{s.2cyl3d}. In addition to being a stronger
result, another reason for proving the proposition here is for completeness of the paper.

\begin{prop} \label{prop:dballs}
Consider any two balls $B_1,B_2$ with radii 1 and whose centers are at distance $r$.
There exists constants $0<c_1,c_2<\infty$ dependent on $d$ but not $r,$ such that for any $r\ge 4,$
\begin{equation}\label{e.dballs}
\frac{c_1}{r^{d-1}}\leq \mu_{d,1}(\CL_{B_1}\cap\CL_{B_2})\leq \frac{c_1}{r^{d-1}}+\frac{c_2}{r^{d+1}}.
\end{equation}
\end{prop}
\noindent
{\bf Remark:} The proposition is easily generalised to hold for any pair of balls of 
arbitrary, fixed radii.

\begin{proof}
Much as in the proof of Proposition \ref{prop:cylinder3d} we have that
\begin{equation}\label{e.exactmeasure}
\mu_{d,1}({\mathcal L}_{B_1}\cap {\mathcal L}_{B_2})=\int_{G(d,1)}
\lambda_{d-1}(\Pi_{L^\bot}(B_1)\cap \Pi_{L^\bot}(B_2))
\nu_{d,1}(dL),
\end{equation}
where $\lambda_{d-1}$ denotes $d-1$-dimensional Lebesgue measure.
Without loss of generality, we let the center of $B_1$ be the origin, while the center of
$B_2$ is $(r,0,\ldots,0).$ Observe that if $L\in G(d,1)$ is written on the form $L=\{t(l_1,...,l_d)\,:\,t\in {\mathbb R}\}$
with $l_1^2+...+l_d^2=1$, then the projection matrix $\Pi_{L^{\bot}}$ is given by
$(\Pi_{L^{\bot}})_{ii}=1-l_i^2$ and $(\Pi_{L^{\bot}})_{ij}=-l_i l_j$ for $i\neq j$.
Let $p_r=p_r(L)=\Pi_{L^\bot}((r,0,\ldots,0)).$
Then straightforward calculations give
\[
p_r=r(1-l_1^2,-l_1l_2,\ldots,-l_1l_d),
\]
so that
\[
|p_r|^2=r^2((1-l_1^2)^2+l_1^2l_2^2+\cdots+l_1^2l_d^2)=r^2(1-2l_1^2+l_1^2(l_1^2+\cdots+l_d^2))=r^2(1-l_1^2).
\]

Observe that $\Pi_{L^\bot}(B_1)$ and $\Pi_{L^\bot}(B_2)$ intersects whenever
$|p_r|\leq 2,$ or equivalently when $|l_1|\geq \sqrt{1-4/r^2}.$
Furthermore, the volume of the lens $\Pi_{L^\bot}(B_1)\cap \Pi_{L^\bot}(B_2)$ is simply
the sum of two spherical caps of height $h=\max(1-|p_r|/2,0).$ The volume of one such spherical
cap is (see for instance \cite{LI}) given by
\[
\frac{1}{2}\kappa_{d}J_{2h-h^2}\left(\frac{d+1}{2},\frac{1}{2}\right).
\]
As before, $\kappa_d$ is the volume of the d-dimensional ball and $J_{2h -h^2}$ denotes a regularized
incomplete beta function. We note that $2h-h^2=\max(1-|p_r|^2/4,0)=\max(1-r^2(1-l_1^2)/4,0),$
so that \eqref{e.exactmeasure} becomes
\begin{equation}\label{eqn:intsphere1}
\kappa_d\int_{G(d,1)}J_{1-r^2(1-l_1^2)/4}\left(\frac{d+1}{2},\frac{1}{2}\right)
I\{|l_1|\geq \sqrt{1-4/r^2}\}\nu_{d,1}(dL).
\end{equation}
%\begin{eqnarray} \label{eqn:intsphere1}
%\lefteqn{\int_{G(d,1)}|\Pi_{L^\bot}(B_1)\cap \Pi_{L^\bot}(B_2)|\nu_{d,1}(dL)}\\
%& & =V_d\int_{G(d,1)}J_{1-r^2(1-l_1^2)/4}\left(\frac{d+1}{2},\frac{1}{2}\right)
%I\{|l_1|\geq \sqrt{1-4/r^2}\}\nu_{d,1}(dL). \nonumber
%\end{eqnarray}
Furthermore,
\[
J_{1-r^2(1-l_1^2)/4}\left(\frac{d+1}{2},\frac{1}{2}\right)
=\frac{\int_0^{1-r^2(1-l_1^2)/4}t^{\frac{d+1}{2}-1}(1-t)^{\frac{1}{2}-1}dt}
{\int_0^{1}t^{\frac{d+1}{2}-1}(1-t)^{\frac{1}{2}-1}dt}.
\]
Let $f(t):=t^{\frac{d+1}{2}-1}(1-t)^{\frac{1}{2}-1}$ and set $D_{d+1}=\int_0^1 f(t)dt.$
We get that \eqref{eqn:intsphere1} equals
\begin{eqnarray} \label{eqn:ballest1}
\lefteqn{\frac{\kappa_d}{D_{d+1}}\int_{G(d,1)} \int_0^1I\{t\leq 1-r^2(1-l_1^2)/4\}f(t)dt
I\{l_1^2\geq 1-4/r^2\}\nu_{d,1}(dL)}\\
& & =\frac{\kappa_d}{D_{d+1}}\int_{G(d,1)} \int_0^1I\{l_1^2 \geq 1-4(1-t)/r^2\}
f(t)dtI\{l_1^2\geq 1-4/r^2\}\nu_{d,1}(dL)\nonumber\\
& & =\frac{\kappa_d}{D_{d+1}}\int_0^1\int_{G(d,1)} I\{l_1^2 \geq 1-4(1-t)/r^2\}
\nu_{d,1}(dL)f(t)dt. \nonumber
\end{eqnarray}
Furthermore, since $\nu_{d,1}(G(d,1))=1,$ we have that (see again \cite{LI})
%\begin{equation}\label{e.caparea}
\begin{equation} \label{eqn:ballest2}
\int_{G(d,1)}{\bf 1}\{|l_1|\geq \sqrt{1-4(1-t)/r^2}\}\nu_{d,1}(dL)
=J_{2h'-h'^2}\left(\frac{d-1}{2},\frac{1}{2}\right),
\end{equation}
%\end{equation}
where $h'=1-\sqrt{1-4(1-t)/r^2}$ so that $2h'-h'^2=4(1-t)/r^2.$
%and $A_d$ denotes
%the surface area of the unit sphere in dimension $d.$
As above, with $D_{d-1}=\int_0^{1}s^{\frac{d-1}{2}-1}(1-s)^{\frac{1}{2}-1}ds,$
we get that
%\[
%I_{4(1-t)/r^2}(\frac{d-1}{2},\frac{1}{2})
%=\frac{1}{D(d-1)}\int_0^{4(1-t)/r^2}s^{\frac{d-1}{2}-1}(1-s)^{\frac{1}{2}-1}ds
%=\frac{1}{D(d-1)}\int_0^{4(1-t)/r^2}s^{\frac{d-1}{2}-1}(1-s)^{\frac{1}{2}-1}ds
%\]
\begin{eqnarray}\label{eqn:ballest3}
\lefteqn{J_{4(1-t)/r^2}\left(\frac{d-1}{2},\frac{1}{2}\right)
=\frac{\int_0^{4(1-t)/r^2}s^{\frac{d-1}{2}-1}(1-s)^{\frac{1}{2}-1}ds}
{\int_0^{1}s^{\frac{d-1}{2}-1}(1-s)^{\frac{1}{2}-1}ds}}\\
& & \geq \frac{1}{D_{d-1}}\int_0^{4(1-t)/r^2}s^{\frac{d-1}{2}-1}ds
=\frac{2}{(d-1)D_{d-1}}\left(\frac{4(1-t)}{r^2}\right)^{(d-1)/2}
=C_1(1-t)^{(d-1)/2}r^{-(d-1)}. \nonumber
\end{eqnarray}
Furthermore, since $r\geq 4$, then for $0\leq s \leq 4(1-t)/r^2\leq 1/4$ (as $0\leq t\leq 1$),
\[
\frac{1}{1-s}=1+\frac{s}{1-s}\leq 1+2s\leq (1+s)^2,
\]
so that $(1-s)^{-1/2}\leq 1+s.$ Hence,
\begin{eqnarray}\label{eqn:ballest4}
\lefteqn{J_{4(1-t)/r^2}\left(\frac{d-1}{2},\frac{1}{2}\right)
\leq \frac{1}{D_{d-1}}\int_0^{4(1-t)/r^2}s^{\frac{d-1}{2}-1}(1+s)ds}\\
& & =C_1(1-t)^{(d-1)/2}r^{-(d-1)}+\frac{2}{(d+1)D_{d-1}}\left(\frac{4(1-t)}{r^2}\right)^{(d+1)/2} \nonumber \\
& & =C_1(1-t)^{(d-1)/2}r^{-(d-1)}+C_2(1-t)^{(d+1)/2}r^{-(d+1)}.\nonumber
\end{eqnarray}
Combining \eqref{eqn:intsphere1}, \eqref{eqn:ballest1}, \eqref{eqn:ballest2} and \eqref{eqn:ballest3}
we get that
\begin{eqnarray*}
\lefteqn{\int_{G(d,1)}\lambda_2(\Pi_{L^\bot}(B_1)\cap \Pi_{L^\bot}(B_2))\nu_{d,1}(dL)}\\
& & \geq r^{-(d-1)}\frac{C_1\kappa_d}{D_{d+1}} \int_0^1 f(t)(1-t)^{(d-1)/2}dt=C_3r^{-(d-1)}.
\end{eqnarray*}
Similarly, combining \eqref{eqn:intsphere1}, \eqref{eqn:ballest1}, \eqref{eqn:ballest2}
and \eqref{eqn:ballest4} we get that
\begin{eqnarray*}
\lefteqn{\int_{G(d,1)}\lambda_2(\Pi_{L^\bot}(B_1)\cap \Pi_{L^\bot}(B_2))\nu_{d,1}(dL)}\\
& & \leq C_3r^{-(d-1)}+r^{-(d+1)}\frac{C_1\kappa_d}{D_{d+1}} \int_0^1 f(t)(1-t)^{(d+1)/2}dt\\
& & =C_3r^{-(d-1)}+C_4r^{-(d+1)}.
\end{eqnarray*}
\end{proof}

\subsection{Measure of the set of lines that intersect a ball and a cylinder}

In this section we estimate the measure of lines that intersect both a ball and a cylinder that are far apart.
We say that a ball $B(x,1)$ and a cylinder $\cc$ is at distance $r,$ if the distance between $x$ and 
the centerline of $\cc$ is $r.$

\begin{prop} \label{prop:ball-line1}
For any $d\geq 3,$ there exist constants $c=c(d)>0$ and $c'=c'(d)<\infty$ such that for all $r\geq 1,$
and $x\in \BR^d,$ $L\in A(d,1)$ at distance $r$ from each other, we have that
\begin{equation}\label{e.ballcylmeas}
c\, r^{-(d-2)}\le \mu_{d,1}({\mathcal L}_{B(x,1)}\cap {\mathcal L}_{{\mathfrak c}(L)})\le c'\, r^{-(d-2)}.
\end{equation}
\end{prop}
\begin{proof}
We begin with the upper bound. By rotation and translation invariance
of $\mu_{d,1}$, we can without loss of generality assume that $x=(r,...,0)$ and $L=\{t e_2\,:\,t\in{\mathbb R}\}$. For
$i\in {\mathbb Z}$, let $B_i:=B((0,i,0,...,0),2)$. Observe that
$${\mathcal L}_{{\mathfrak c}(L)}\subset \bigcup_{i\in {\mathbb Z}} {\mathcal L}_{B_i}$$ so that
\begin{equation}\label{e.linsetincl}
{\mathcal L}_{B(x,1)}\cap {\mathcal L}_{{\mathfrak c}(L)}
\subset \bigcup_{i\in {\mathbb Z}} \left({\mathcal L}_{B(x,1)}\cap {\mathcal L}_{B_i}\right).
\end{equation}
We now get for $r\ge 1$ that
\begin{eqnarray*}
\lefteqn{\mu_{d,1}({\mathcal L}_{B(x,1)}\cap {\mathcal L}_{{\mathfrak c}(L)})
\stackrel{~\eqref{e.linsetincl}}{\le} \sum_{i=-\infty}^{\infty}\mu_{d,1}({\mathcal L}_{B(x,1)}\cap {\mathcal L}_{B_i})
\stackrel{~\eqref{e.dballs}}{\le} c\sum _{i=-\infty}^{\infty} (r^2+i^2)^{-(d-1)/2}}\\
 & &= c r^{-(d-1)}\sum _{i=-\infty}^{\infty} (1+(i/r)^2)^{-(d-1)/2}\le c' r^{-(d-1)} \int_{-\infty}^{\infty} (1+(x/r)^2)^{-(d-1)/2}\, dx\\ & &=c' r^{-(d-2)}\int_{-\infty}^{\infty} (1+y^2)^{-(d-1)/2}\, dy=c'' r^{-(d-2)},
\end{eqnarray*}
where the integral in the last step is convergent since $d\ge 3$. This finishes the proof of the 
upper bound in~\eqref{e.ballcylmeas}, and we proceed with the lower bound.

For proof-technical reasons we now assume that $r\ge 10$.
For $i\in \{2,3,...,\lfloor r \rfloor\}$, let $D_i:=B((0,i,...,0),1/8)$. Observe that
$$\bigcup_{i=2}^{\lfloor r \rfloor}{\mathcal L}_{D_i}\subset {\mathcal L}_{{\mathfrak c}(L)}$$ so that

\begin{equation}\label{e.linsetincl2}
\bigcup_{i=2}^{\lfloor r \rfloor}({\mathcal L}_{D_i}\cap {\mathcal L}_{B(x,1)})
\subset {\mathcal L}_{{\mathfrak c}(L)}\cap {\mathcal L}_{B(x,1)}.
\end{equation}
We will now show that
\begin{equation}\label{e.linedisjoints}
({\mathcal L}_{D_i}\cap {\mathcal L}_{B(x,1)})_{i=2}^{\lfloor r \rfloor}
\mbox{ is a sequence of pairwise disjoint sets of lines.}
\end{equation}
Let $i,j\in\{2,...,\lfloor r \rfloor\}$ where $i\neq j$ and assume that
\begin{equation}\label{e.lcontr1}
L_1\in {\mathcal L}_{D_i} \cap {\mathcal L}_{D_j}
\end{equation}
and that
\begin{equation}\label{e.lcontr2}
L_1\in {\mathcal L}_{D_i}\cap {\mathcal L}_{B(x,1)}.
\end{equation}
As usual, we write $L_1$ on the form $L_1=\{t(k_1,...,k_d)\,:\,t\in {\mathbb R}\}+v$ for some $v\in {\mathbb R}^d$. We observe that if~\eqref{e.lcontr1} holds, then
%\begin{equation}\label{e.k1}
\[
\frac{k_1}{k_2} \ge \frac{-2/8}{|i-j|-2/8}\ge \frac{-2/8}{1-2/8}\ge -\frac{1}{3},
\]
%\end{equation}
while for~\eqref{e.lcontr2} to be satisfied, then
%\begin{equation}\label{e.k2}
\[
\frac{k_1}{k_2}\le -\frac{r-1-1/8}{j+1+1/8}\le -\frac{r-1-1/8}{r+1+1/8}\le-\frac{10-1-1/8}{10+1+1/8}= -\frac{71}{89}.
\]
%\end{equation}
We conclude that~\eqref{e.lcontr1} and~\eqref{e.lcontr2} cannot both hold, which proves~\eqref{e.linedisjoints}.

Proceeding, we have that
%\begin{equation}\label{e.lboundpunch}
\begin{eqnarray*}
\lefteqn{\mu_{d,1}({\mathcal L}_{{\mathfrak c}(L)}\cap {\mathcal L}_{B(x,1)})
\stackrel{~\eqref{e.linsetincl2}}{\ge}
\mu_{d,1}(\cup_{i=2}^{\lfloor r \rfloor}({\mathcal L}_{D_i}\cap {\mathcal L}_{B(x,1)}))
\stackrel{~\eqref{e.linedisjoints}}{=}
\sum _{i=2}^{\lfloor r \rfloor} \mu_{d,1}({\mathcal L}_{D_i}\cap {\mathcal L}_{B(x,1)})}\\ & &
\stackrel{~\eqref{e.dballs}}{\ge}c \sum _{i=2}^{\lfloor r \rfloor} (r^2+i^2)^{-(d-1)/2}\ge c \sum _{i=2}^{\lfloor r \rfloor}( 2 r^2)^{-(d-1)/2}= c' (\lfloor r\rfloor -1) r^{-(d-1)}\ge c'' r^{-(d-2)},
\end{eqnarray*}
finishing the proof of the proposition in the case $r\ge 10$. The full statement follows easily.
\end{proof}

\section{Proof of Theorem~\ref{t.maind}, the lower bound when $d\geq 4$}\label{s.lbound}
In this section we prove the following theorem.
\begin{thm} \label{thm:lbound}
For any $d\geq 4,$ $\BP[\diam(\CC)\geq d-2]=1.$
\end{thm}
 As a key step, we first show that the probability that
two points $x$ and $y$ in ${\mathbb R}^d$ are connected via a sequence of at most $d-1$ cylinders
tends to $0$ as $|x-y|\to \infty$, see Proposition~\ref{p.lbound} below.
We will think of the integer lattice ${\mathbb Z}^d$ as a subset ${\mathbb R}^d$, embedded in the natural way.

For each $y=(y_1,...,y_d)\in {\mathbb R}^d$ let
$\lfloor y \rfloor:=(\lfloor y_1 \rfloor,...,\lfloor y_d \rfloor)\in {\mathbb Z}^d$.

We will need the following lemma, which (as remarked in \cite{RS}) follows from $(1.38)$ of
Proposition $1.7$ in \cite{HHS}.
\begin{lem}\label{l.convulest}
For any positive integer $n<d$ and any $z_0,z_n\in {\mathbb Z}^d$ there exists a constant $c=c(u,d)<\infty$ such that 
\begin{equation}\label{e.mconvulest}
\sum_{z_1,...,z_{n-1}\in {\mathbb Z}^d}\prod_{i=0}^{n-1} \min(1,|z_i-z_{i+1}|^{-(d-1)})\le c |z_0-z_n|^{-(d-n)}.
\end{equation}
\end{lem}

For $x,y\in {\mathbb R}^d$ and $n\ge 1$, let $A_n(x,y)$ be the event that there exist
\emph{distinct} lines $L_1,...,L_n\in \omega$ such that
\begin{enumerate}
\item[1.] $x\in {\mathfrak c}(L_1)$ and $y\in {\mathfrak c}(L_n)$.
\item[2.] ${\mathfrak c}(L_i)\cap {\mathfrak c}(L_{i+1})\neq \emptyset$ for $i=1,...,n-1$.
\end{enumerate}
In addition, let
$$\tilde{A}_n(x,y)=\cup_{i=1}^n A_i(x,y).$$
We can now state the first result of this section:
\begin{prop}\label{p.lbound}
For $d\geq 3,$ $n\in \{1,...,d-1\}$ there exists a constant $c=c(u,d)<\infty$ such that
for any $x,y\in {\mathbb R}^d$ with $|x-y|\ge 2 d$, 
%\begin{equation}
\[
{\mathbb P}[\tilde{A}_n(x,y)]\le c |x-y|^{-(d-n)}.
\]
%\end{equation}
\end{prop}

\begin{proof}
The proof follows the first part of the proof of Theorem $1$ in \cite{RS} closely.
Recall that we think of the integer lattice ${\mathbb Z}^d$ as embedded in ${\mathbb R}^d$.
Fix $n\in \{1,...,d-1\}$ and $x,y\in {\mathbb R}^d$ where $|x-y|\ge 2 d$.

For $v,w\in {\mathbb Z}^d$, let
$$T(v,w):={\mathcal L}_{B(v,\sqrt{d}+1)}\cap {\mathcal L}_{B(w,\sqrt{d}+1)}$$
and introduce the event
$$E(v,w):=\{\omega\left(T(v,w)\right)\ge 1\}.$$
For $z_0,\ldots,z_n \in \BZ^d$ we let $E(z_0,z_1)\circ E(z_1,z_2)\circ...\circ E(z_{n-1},z_n)$
denote the event that there exists distinct lines $L_1,\ldots,L_n$ such that $L_i\in E(z_{i-1},z_i)$
for every $i=1,\ldots,n.$
If $A_n(x,y)$ occurs, then there exist distinct lines $L_1,...,L_n$ in $\omega$ and
points $x_1,...,x_{n-1}\in {\mathbb R}^d$ such that
$$L_i\in {\mathcal L}_{B(x_{i-1},1)}\cap {\mathcal L}_{B(x_{i},1)}\mbox{ for }i=1,..,n,$$
where we put $x_0:=x$ and $x_n:=y$.
Since $|x-\lfloor x \rfloor|\le \sqrt{d}$ for any $x\in {\mathbb R}^d$, it follows that we also have
\[
L_i\in {\mathcal L}_{B(\lfloor x_{i-1}\rfloor,1+\sqrt{d})}
\cap {\mathcal L}_{B(\lfloor x_{i} \rfloor,1+\sqrt{d})}\mbox{ for }i=1,..,n.
\]
Therefore, we have shown the inclusion

\begin{equation}\label{e.Aincl}
A_n(x,y)\subset \bigcup_{z_1,...,z_{n-1}\in {\mathbb Z}^d}E(z_0,z_1)\circ E(z_1,z_2)\circ...\circ E(z_{n-1},z_n),
\end{equation}
where we let $z_0:=\lfloor x \rfloor$ and $z_n:= \lfloor y \rfloor$.
Let $\omega_{\neq}^n$ denote the set of all $n-$tuples of distinct lines $L_1,...,L_n$ in $\omega$. Then we have
\begin{equation}\label{e.indicatoreq}
{\bf 1}\{E(z_0,z_1)\circ...\circ E(z_{n-1},z_n)\}\le \sum_{\omega_{\neq}^n} \prod_{i=1}^{n}{\bf 1}\{L_i \in T(z_{i-1},z_{i})\}.
\end{equation}
Now a union bound together with~\eqref{e.Aincl} and~\eqref{e.indicatoreq} implies
\begin{equation}\label{e.Aupper1}
{\mathbb P}[A_n(x,y)]\le \sum_{z_1,...,z_{n-1}\in {\mathbb Z}^d}
{\mathbb E}\left[\sum_{\omega_{\neq}^n} \prod_{i=1}^{n}{\bf 1}\{L_i \in T(z_{i-1},z_{i})\}\right].
\end{equation}
According to the Slivnyak-Mecke formula (see \cite{SW} Corollary $3.2.3$), the expectation on the right hand side
of~\eqref{e.Aupper1} equals
\begin{eqnarray}\label{e.SlivMeck}
\lefteqn{\prod_{i=1}^{n}{\mathbb E}[\omega(T(z_{i-1},z_{i}))]
= u^n \prod_{i=1}^{n} \mu_{3,1}(T(z_{i-1},z_{i}))}\\
& &\le c(u,d)\prod_{i=1}^{n}\min(1,|z_{i-1}-z_{i}|^{-(d-1)}),\nonumber
\end{eqnarray}
where we applied Proposition~\ref{prop:dballs} in the last inequality.
From~\eqref{e.Aupper1} and~\eqref{e.SlivMeck} we get
\begin{eqnarray}\label{e.Aupper2}
\lefteqn{{\mathbb P}[A_n(x,y)]
\le c(u,d)\sum_{z_1,...,z_{n-1}\in {\mathbb Z}^d}\prod_{i=1}^{n} \min(1,|z_{i-1}-z_{i}|^{-(d-1)})}\\
& &\stackrel{~\eqref{e.mconvulest}}{\le} c(u,d) |\lfloor x \rfloor-\lfloor y \rfloor|^{-(d-n)}\le c'(u,d)|x-y|^{-(d-n)},\nonumber
\end{eqnarray}
whenever $n\in \{1,...,d-1\}$. Finally we get
\begin{eqnarray*}
\lefteqn{{\mathbb P}[\tilde{A}_n(x,y)] \le \sum_{k=1}^n{\mathbb P}[A_k(x,y)]}\\
& &\stackrel{~\eqref{e.Aupper2}}{\le} \sum_{k=1}^n c(u,d) |x-y|^{-(d-k)} \le c'(u,d) |x-y|^{-(d-n)}.
\end{eqnarray*}
\end{proof}

For $n\leq d-1,$ consider the event
\[
\bigcup_{x,y\in{\mathbb R}^d}\left(\tilde{A}_n(x,y)^c\cap \{x,y\in{\mathcal C}\}\right).
\]
In words: there exist points $x,y\in\CC$ which are not connected via any sequence of
$n$ cylinders if $n\le d-1$.
From Proposition~\ref{p.lbound} it is quite intuitive that the probability of this event should be 1.
We prove this
in full detail in the next corollary, which is in the spirit of the final part of the proof of Theorem
$2.1(ii)$ in \cite{LT}.
\begin{cor}\label{c.lbound}
For any $n\le d-1$ we have
\begin{equation}\label{e.asevent}
{\mathbb P}\left[\bigcup_{x,y\in{\mathbb R}^d}\left(\tilde{A}_n(x,y)^c\cap \{x,y\in{\mathcal C}\}\right)\right]=1.
\end{equation}
\end{cor}

\begin{proof}
Fix $n\le d-1$ and identify ${\mathbb Z}$ with the points along the $e_1$-axis with integer coordinates.
For $R\ge 1$  let $K^1_R$ denote the set ${\mathbb Z}\cap [1,R+1],$ $K^2_R$ denote the
set ${\mathbb Z}\cap [e^R,e^R+R]$ and define
\[
H_R:=\bigcup_{x\in K^1_R,y\in K^2_R}\left(\tilde{A}_n(x,y)^c\cap \{x,y\in {\mathcal C}\}\right).
\]
We will show that
\[
{\mathbb P}[H_R]\stackrel{R\to \infty}{\to}1,
\] which implies
\[
{\mathbb P}[\cup_{R\ge 1} H_R]=1.
\]
Then~\eqref{e.asevent} follows since
\[
\bigcup_{R\ge 1} H_R \subset
\bigcup_{x,y\in{\mathbb R}^d}\left(\tilde{A}_n(x,y)^c\cap \{x,y\in{\mathcal C}\}\right).
\]

Let $E^1_R$ be the event that there is no pair $x\in K_R^1$ and $y\in K_R^2$ for which $x,y\in {\mathcal C}$.
That is, we let
\begin{equation}\label{e.e1expr}
E^1_R:=\bigcap_{x\in K_R^1} \bigcap_{y\in K_R^2}\{x,y\in {\mathcal C}\}^c
=\{{\mathcal C}\cap K^1_R=\emptyset\}\cup \{{\mathcal C}\cap K^2_R=\emptyset\}.
\end{equation}
Also, introduce the event
\[
E_R^2:=\bigcup_{x\in K_R^1} \bigcup_{y\in K_R^2} \tilde{A}_n(x,y),
\]
which is the event that there exists $x\in K_R^1$ and $y\in K_R^2$ such that they are connected
via at most $n$ cylinders.
We have
\begin{equation}\label{e.HReq}
H_R^c  =\bigcap_{x\in K_R^1} \bigcap_{y\in K_R^2} (\{x,y\in {\mathcal C}\}\cap \tilde{A}_n(x,y)^c)^c= \bigcap_{x\in K_R^1} \bigcap_{y\in K_R^2}\left(\{x,y \in {\mathcal C}\}^c\cup \tilde{A}_n(x,y)\right).
\end{equation}
From~\eqref{e.HReq} we see that
\begin{equation*}
H_R^c \cap E_R^1 = E^1_R\mbox{ and } H_R^c \cap (E_R^1)^c \subset E_R^2.
\end{equation*}
The second inclusion follows since if $(E_R^1)^c$ occurs, then there must exist $x\in K_R^1$
and $y\in K_R^2$ such that $x,y\in \CC,$ and for $H_R^c$ to occur, $\tilde{A}_n(x,y)$
must occur for these $x,y.$
Hence
\[
H_R^c\subset E_R^1\cup E_R^2.
\]
We now argue that
\begin{equation}\label{e.eprobs}
\lim_{R\to \infty}{\mathbb P}[E^1_R]=0.
\end{equation}
For $i\in {\mathbb Z}$ let $B_i$ be the ball of radius 1 centered at $(i,0,\ldots,0)$ and 
$F_i={\mathcal L}_{B_i}\setminus ({\mathcal L}_{B_{i-1}}\cup {\mathcal L}_{B_{i+1}})$.
It is straightforward to see that if $L\in F_i\subset {\mathcal L}_{B_i}$,
then $L\notin {\mathcal L}_{B_j}$ for every $j\neq i$. Hence $(F_i)_{i\in {\mathbb Z}}$ is
a sequence of disjoint sets of lines. It is also easy to see that $\mu_{d,1}(F_i)=c_1(d)>0$.
We now get
\begin{eqnarray*}
\lefteqn{
{\mathbb P}[{\mathcal C}\cap K^1_R=\emptyset]
={\mathbb P}[\cap_{i\in K_R^1}\{\omega({\mathcal L}_{B_i})=0\}]\le {\mathbb P}[\cap_{i\in K_R^1}\{\omega(F_i)=0\}]}\\
& &=\Pi_{i\in {K_R^1}} {\mathbb P}[\omega(F_i)=0]= e^{-|K_R^1| c_1(d)} =e^{-c R}.
\end{eqnarray*}
In the same way, we get ${\mathbb P}[{\mathcal C}\cap K^2_R=\emptyset]\le e^{-c R}$
so~\eqref{e.e1expr}, a union bound, and letting $R\to \infty$ gives~\eqref{e.eprobs}.

For the event $E^2_R$ we have
\[
{\mathbb P}[E^2_R]\le \sum_{x\in K_R^1}\sum_{y\in K_R^2}
{\mathbb P}[\tilde{A}_n(x,y)]\stackrel{~\eqref{e.Aupper2}}{\le} c R^2 e^{-R}\to 0,
\]
as $R\to \infty$.
Hence,
\[
\lim_{R\to \infty}{\mathbb P}[H_R^c]
\le \lim_{R\to \infty}{\mathbb P}[E^1_R]+\lim_{R\to \infty}{\mathbb P}[E_R^2]=0,
\]
as required.
\end{proof}

{\bf Proof of Theorem \ref{thm:lbound}:}
Let $k\le d-3$. According to Corollary~\ref{c.lbound}, we can
a.s.\ find $x,y\in {\mathcal C}$ such that $x$ and $y$ are not connected via any sequence of $k+2$
cylinders. Since $x,y\in {\mathcal C}$, this means that there is a cylinder ${\mathfrak c}_1\in\omega$
(which contains $x$) and a cylinder ${\mathfrak c}_2\in\omega$ (which contains $y$) such that
${\mathfrak c}_1$ and ${\mathfrak c}_2$ are not connected via any sequence of $k$ cylinders.
\fbox{}\\

\section{Proof of Theorem~\ref{t.maind}, the upper bound when $d\geq 4$}\label{s.maind}

In this section we prove 
\begin{thm} \label{thm:ubound}
For any $d\geq4,$ $\BP(\diam(\CC)\leq d-2)=1.$
\end{thm}
\noindent
Obviously, Theorem ~\ref{t.maind} follows from Theorems \ref{thm:maind=3}, \ref{thm:lbound} and \ref{thm:ubound}.
Only the proof of Theorem \ref{thm:ubound} remains.

\medskip

The proof of Theorem \ref{thm:ubound} is fairly long. In order to facilitate the
reading, we will try to provide a short intuitive and very informal description
of the main underlying idea. We let a cylinder-path of length $k$ from $\cc_1$ to $\cc_2$
be a collection $\cc^1,\ldots,\cc^{k}$ of cylinders such that $\cc_1\cap \cc^1\neq \emptyset,$
$\cc^1 \cap \cc^2\neq \emptyset$ and so on. Assuming Theorem~\ref{thm:ubound}, there
should be plenty of such cylinder-paths from $\cc_1$ to $\cc_2$ using $d-2$ cylinders.
We will therefore look for collections of boxes $B_1,\ldots,B_{d-3}$ (of small sidelength) such
that $\cc^i$ and $\cc^{i+1}$ "meets" in $B_i,$ that is $\cc^i\cap\cc^{i+1}\cap B_i\neq \emptyset.$
Finding such collections are complicated by the longe-range dependencies of the line-process $\omega.$
Therefore, we will have to be very careful in the way 
we look for the boxes, in order to have enough independence for the proof to work. This will be done
by considering particular cylinder-paths on a sequence of different scales. \\

Before presenting the proof, we start with some definitions. 
For simplicity, we assume in this section that the radius of a cylinder is $\sqrt{d},$ 
the reason for this will be clear shortly and can be made without loss of generality.

Consider two arbitrary cylinders $\cc_1,\cc_2$ with 
centerlines $L_1,L_2\in A(d,1)$
respectively. Since for any two lines in $G(d,1)$, there is a plane that they belong to, we can without
loss of generality (due to the invariances of $\mu_{d,1}$) assume that $L_1=\{te_1:-\infty<t<\infty\}$ and that
$L_2=\{p+t(l_1,l_2,0,\ldots,0):-\infty<t<\infty\}$ where
$p=(0,0,p_3,p_4,\ldots,p_d).$ 

For any integers $m,R$ consider the boxes $B^1_{R^m},B^2_{R^m},\ldots,B^{d-3}_{R^m}$ where
$B^i_{R^m}=q_{i,m}+[-{R^m}/2,{R^m}/2]^d,$ $q_{i,m}=p+N{R^m}e_{i+3}$ and $N=10d+1.$ 
The reason for this choice of $N$ will become clear later. We will assume throughout that 
$R\geq 2 \max_{i=1,\ldots,d}|p_i|.$

We can tile the boxes $B^i_{R^m}$ in the canonical way with smaller boxes of sidelength 1.
We denote such boxes by $B_{i,m},$ that is $B_{i,m}\subset B_{R^m}^i$.  
%In fact, we will sometimes drop the $m$ from the notation when it is a fixed number.
%DO WE REALLY WANT THIS OR SHOULD WE DROP THE $m$ FROM $B_{i,m},$?
Note that if the centerlines $L_a,L_b$
of cylinders ${\mathfrak c}_a,{\mathfrak c}_b$ both intersect a box $B$ of sidelength 1, then
since the radii of the cylinders are $\sqrt{d},$ we have ${\mathfrak c}_a\cap{\mathfrak c}_b\neq \emptyset.$
This is the reason for our choice of radius.

For any two sets $E_1,E_2\subset \BR^d,$ we let $E_1 \leftrightarrow E_2$ denote the event that
the Poisson process $\omega$ includes an element $L$ in the set $\CL_{E_1}\cap\CL_{E_2}.$ We will say that
$E_1,E_2$ are connected, and that $L$ connects $E_1$ and $E_2.$
Furthermore, we let $E_1 \stackrel{n}{\leftrightarrow} E_2$ denote the event that there are
exactly $n$ such connecting lines. It will greatly facilitate our analysis to consider disjoint 
parts of the cylinders $\cc_1,\cc_2$. Therefore, we define for every $m\geq 1,$
\[
\cc_{1,m}:=\{x\in \cc_1:R^{m-1}/2\leq d(\Pi_{L_1}(x),o)<R^m/2-10\sqrt{d}\},
\]
and
\[
\cc_{2,m}:=\{x\in \cc_2:R^{m-1}/2\leq d(\Pi_{L_2}(x),p)<R^m/2-10\sqrt{d}\}.
\]
Let $\vec{B}_m:=(B_{1,m},\ldots,B_{d-3,m})$ and let $\CP_m(\vec{B}_m)$ be shorthand for
$\{\cc_{1,m}\leftrightarrow B_{1,m}\leftrightarrow \ldots\leftrightarrow B_{d-3,m}\leftrightarrow\cc_{2,m}\}.$
We define
\[
X_{R,m}=\sum_{\vec{B}_m}I(\CP_m(\vec{B}_m)),
\]
where the sum is over all choices of $\vec{B}_m$ and $I$ is an indicator function.
Obviously, if $X_{R,m}>0,$ then $\cc_1,\cc_2$ are connected via a cylinder-path of length $d-2.$

We will prove the following two lemmas.
\begin{lem} \label{lemma:bound}
There exists a constant $c=c(u,d)>0$ such that for all $R$ 
large enough, $\BP(X_{R,m}>0)\geq c$ for every $m\geq 1.$
\end{lem}
\noindent
\begin{lem} \label{lemma:geometric1}
For any $R$ large enough, the sets
\[
%\left\{
\CL_{\cc_{1,1}}\cap \CL_{B^1_{R^1}},\CL_{B^1_{R^1}}\cap \CL_{B^2_{R^1}},\ldots,
\CL_{B^{d-4}_{R^1}}\cap \CL_{B^{d-3}_{R^1}},\CL_{B^{d-3}_{R^1}}\cap \CL_{\cc_{2,1}},
\CL_{\cc_{1,2}}\cap \CL_{B^1_{R^2}},\ldots,\CL_{\cc_{1,m}}\cap \CL_{B^1_{R^m}},\ldots
%\right\}
\]
are mutually disjoint. 
\end{lem}
Before presenting the proofs of these lemmas, we will show how Theorem ~\ref{thm:ubound} 
follows from them. \\

\noindent 
{\em Proof of Theorem~\ref{thm:ubound}} 
We observe that Lemma \ref{lemma:geometric1} implies that $\{X_{R,m}\}_{m \geq 1}$
is a sequence of independent random variables, as $X_{R,m}$ is defined only in terms of the sets
\begin{equation} \label{eqn:disjointsets}
\CL_{\cc_{1,m}}\cap \CL_{B^1_{R^m}},\CL_{B^1_{R^m}}\cap \CL_{B^2_{R^m}},\ldots,
\CL_{B^{d-4}_{R^m}}\cap \CL_{B^{d-3}_{R^m}},\CL_{B^{d-3}_{R^m}}\cap \CL_{\cc_{2,m}}.
\end{equation}
From this and Lemma \ref{lemma:bound} it follows via Borel-Cantelli, that 
\[
\BP[\exists m\geq1: X_{R,m}>0]=1,
\]
so that $\cc_1$ and $\cc_2$ are connected via at most $d-2$ cylinders.
It only remains to show that a.s every pair of cylinders $\cc_1,\cc_2\in\omega$ are 
also connected by at most $d-2$ cylinders.
However, this is completely analogous to the proof in the case $d=3$, so we will be brief. 
For lines $L_1,L_2\in A(d,1)$, let
\[
E(L_1,L_2)=\{\mbox{Cdist}({\mathfrak c}(L_1),{\mathfrak c}(L_2))\le d-2\},
\]
and let
\[
D=\cap_{(L_1,L_2)\in \omega_{\neq}^2}E(L_1,L_2).
\]
As in the proof of the case $d=3$, we can show that $\BP[D]=1,$ which implies
the theorem.
\fbox{}\\

We proceed by proving Lemma \ref{lemma:geometric1} as it will also be useful in proving Lemma \ref{lemma:bound}.
\\

\noindent
{\em Proof of Lemma \ref{lemma:geometric1}.}
Throughout the proof, we keep in mind that $|p_i|\le R/2$ for $i=1,...,d$. The lemma will follow if we show the following six statements for $1\le i,j\le d-4$, $m,n\ge 1$ and $R$ large enough:
\begin{equation}\label{e.dis1}
\CL_{\cc_{1,m}}\cap \CL_{B^1_{R^m}}\mbox{ and }\CL_{\cc_{1,n}}\cap \CL_{B^1_{R^{n}}}\mbox{ are disjoint when }m\neq n.
\end{equation}
\begin{equation}\label{e.dis12}
\CL_{\cc_{2,m}}\cap \CL_{B^{d-3}_{R^m}}\mbox{ and }\CL_{\cc_{2,n}}\cap \CL_{B^{d-3}_{R^{n}}}\mbox{ are disjoint when }m\neq n.
\end{equation}
\begin{equation}\label{e.dis2}
\CL_{B^i_{R^m}}\cap \CL_{B^{i+1}_{R^m}}\mbox{ and }\CL_{B^j_{R^n}}\cap \CL_{B^{j+1}_{R^n}}\mbox{ are disjoint unless }m=n\mbox{ and }i=j.
\end{equation}
\begin{equation}\label{e.disCBBB}
\CL_{\cc_{1,m}}\cap \CL_{B_{R^m}^1}\mbox{ and }\CL_{B_{R^n}^i}\cap \CL_{B_{R^n}^{i+1}}\mbox{ are disjoint.}
\end{equation}
\begin{equation}\label{e.c2dis}
\CL_{\cc_2,m}\cap \CL_{B^{d-3}_{R^m}}\mbox{ and }\CL_{B^i_{R^n}}\cap\CL_{B^{i+1}_{R^n}} \mbox{ are disjoint.}
\end{equation}
\begin{equation}\label{e.disCBCB}
\CL_{\cc_{1,n}}\cap \CL_{B^1_{R^n}}\mbox{ and }\CL_{\cc_{2,m}}\cap \CL_{B^{d-3}_{R^m}}\mbox{ are disjoint.}
\end{equation}

We start with~\eqref{e.dis1}. It suffices to show that ${\CL_{\cc_{1,m}}\cap }\CL_{\cc_{1,n}}$
and $\CL_{\cc_{1,m}}\cap \CL_{B^i_{R^m}}$ are disjoint when $m\neq n$.
Suppose that $L\in \CL_{\cc_{1,m}}\cap \CL_{\cc_{1,n}}$. Let $(k_1,...,k_d)$ be a directional vector of $L$. 
Observe that since $L$ intersects both $\cc_{1,m}$ and $\cc_{1,n}$, we have that for some $z\in \cc_{1,m}$ 
and some $z'\in \cc_{1,n}$,
$$\frac{k_{4}}{k_1}=\frac{z_{4}-z_{4}'}{z_1-z_1'}.$$
Hence, (recall that the radii of the cylinders are $\sqrt{d}$)
\begin{equation}\label{e.slope1}
\left|\frac{k_{4}}{k_1}\right|=\left|\frac{z_{4}-z_{4}'}{z_1-z_1'}\right| 
\le \frac{\sqrt{d}-(-\sqrt{d})}{10\sqrt{d}}=1/5.
\end{equation}
Now suppose also that $L\in \CL_{\cc_{1,m}}\cap \CL_{B^1_{R^m}}$. 
We have that for some $z\in\cc_{1,m}$ and some $z'\in B_{R^m}^1$, 
\begin{equation}\label{e.slope2}
\left|\frac{k_{4}}{k_1}\right|= \left|\frac{z_{4}-z'_{4}}{z_1-z_1'}\right|
\ge\left|\frac{p_{4}+(N-1/2)R^m-\sqrt{d}}{-R^m/2-(R^m/2-10\sqrt{d})}\right|\ge N-2,
\end{equation}
provided that $R$ is large enough. Since~\eqref{e.slope1} and~\eqref{e.slope2} cannot both hold, we get~\eqref{e.dis1}. 

The proof of~\eqref{e.dis12} is similar to, but just slightly more technical than the proof of~\eqref{e.dis1} 
since ${\cc_2}$ does not run along a coordinate axis. The details are left to the reader.

Next we establish~\eqref{e.dis2}. To this end, suppose that 
$L_1\in \CL_{B^i_{R^m}}\cap \CL_{B^{i+1}_{R^m}}$, $L_2 \in \CL_{B^j_{R^n}}\cap \CL_{B^{j+1}_{R^n}}$ 
and $L_3\in \CL_{B^i_{R^m}}\cap \CL_{B^j_{R^n}}$. We will show that $L_1$, $L_2$ and $L_3$ cannot all 
be the same line, by showing that at least one of their corresponding directional vectors is linearly 
independent of the two others. Then~\eqref{e.dis2} follows. Let $x\in L_1 \cap B^i_{R^m}$, 
$x'\in L_1\cap B^{i+1}_{R^m}$, $y\in L_2\cap B^j_{R^n}$, $y'\in L_2\cap B^{j+1}_{R^n}$, $z\in L_3 \cap B^i_{R^m}$ 
and $z'\in L_3 \cap B^j_{R^n}$. Let $v_1=x-x'$, $v_2=y-y'$ and $v_3=z-z'$. Then $v_i$ is a directional vector 
of $L_i$. Observe that for some $\alpha,\alpha'\in [-R^m/2,R^m/2]^d$
$$v_1=(\alpha_1-\alpha_1',...,\alpha_{i+3}-\alpha_{i+3}'+N R^m,\alpha_{i+4}-\alpha_{i+4}'-N R^m,...,\alpha_d-\alpha_d'),$$
and for some $\beta,\beta'\in [-R^n/2,R^n/2]^d$,
$$v_2=(\beta_1-\beta_1',...,\beta_{j+3}-\beta_{j+3}'+N R^n,\beta_{j+4}-\beta_{j+4}'-N R^n,...,\beta_d-\beta_d').$$ 
We will only make use of the vector $v_3$ in the case $i=j$. If $i=j$, then for some $\gamma\in [-R^m/2,R^m/2]^d$ 
and $\gamma'\in [-R^n/2,R^n/2]^d,$
$$v_3=(\gamma_1-\gamma_1',...,\gamma_{i+3}-\gamma_{i+3}'+N R^m,\gamma_{i+4}-\gamma_{i+4}'-N R^n,...,\gamma_d-\gamma_d').$$
We will now consider different cases.\\
{\bf Case $i\neq j$, $m,n$ arbitrary:} Without loss of generality, suppose $i>j$. Then
\begin{equation}\label{e.vdir4}
\left|\frac{(v_2)_{i+4}}{(v_1)_{i+4}}\right|=\left|\frac{\beta_{i+4}-\beta_{i+4}'}{\alpha_{i+4}-\alpha_{i+4}'-N R^m}\right|\le \frac{R^n}{N R^m-R^m}=\frac{R^{n-m}}{N-1}.
\end{equation} 
On the other hand
\begin{equation}\label{e.vdir5}
\left|\frac{(v_2)_{j+3}}{(v_1)_{j+3}}\right|=\left|\frac{\beta_{j+3}-\beta_{j+3}'+N R^n}{\alpha_{j+3}-\alpha_{j+3}'}\right|\ge \frac{N R^n-R^n}{R^m}=(N-1) R^{n-m}
\end{equation}
From~\eqref{e.vdir4} and~\eqref{e.vdir5} it follows that 
\begin{equation}\label{e.vdir6}
\left|\frac{(v_2)_{i+4}}{(v_1)_{i+4}}\right|\neq \left|\frac{(v_2)_{j+3}}{(v_1)_{j+3}}\right|,
\end{equation} 
implying that $v_1$ and $v_2$ are linearly independent. Hence, $L_1$ and $L_2$ are different lines.
{\bf Case $i=j$, $m\neq n$:}
Without loss of generality assume that $n>m$. We get
\begin{equation}\label{e.vdir1}
\left|\frac{(v_2)_{j+4}}{(v_3)_{j+4}}\right|=\left|\frac{\beta_{j+4}-\beta_{j+4}'-N R^n}{\gamma_{j+4}-\gamma_{j+4}'-N R^n}\right|\le \frac{N R^n+R^n}{N R^n-R^n/2-R^m/2}\le \frac{N+1}{N-1},
\end{equation}
using $n>m$ in the last inequality. We also have
\begin{equation}\label{e.vdir2}
\left|\frac{(v_2)_{j+3}}{(v_3)_{j+3}}\right|=\left|\frac{\beta_{j+3}-\beta_{j+3}'+N R^n}{\gamma_{j+3}-\gamma_{j+3}'+N R^m}\right|\ge \frac{N R^n-R^n}{N R^m+R^m/2+R^n/2}\ge \frac{(N-1)R^n}{N R^m+R^n}\ge 0.9 (N-1),
\end{equation}
when $R$ is large enough, since $n>m$. Hence, when $R$ is large and by the choice of $N$,
\begin{equation}\label{e.vdir3}
\left|\frac{(v_2)_{j+4}}{(v_3)_{j+4}}\right|\neq \left|\frac{(v_2)_{j+3}}{(v_3)_{j+3}}\right|.
\end{equation}
It follows that $v_2$ and $v_3$ are linearly independent for $R$ large enough, implying that $L_2$ and $L_3$ are not the same line. 

We move on to show~\eqref{e.disCBBB}. Let $L_4\in \CL_{\cc_{1,m}}\cap \CL_{B_{R^m}^1}$ and $L_5\in\CL_{B_{R^n}^i}\cap \CL_{B_{R^n}^{i+1}}$. Let $x\in L_4 \cap \cc_{1,m}$, $x'\in L_4 \cap B_{R^m}^1$, $y\in L_5\cap  B_{R^n}^i$ and $y'\in L_5\cap B_{R^n}^{i+1}$. Let $v_4=x-x'$ and $v_5=y-y'$ be directional vectors for $L_4$ and $L_5$ respectively. Then, since $\cc_{1,m}\subset [-R^m/2,R^m/2]^d$, for some $\alpha,\alpha'\in [-R^m/2,R^m/2]^d$,
$$v_4=(\alpha_1-p_1-\alpha'_1,...,\alpha_4-p_4-\alpha'_4-N R^m,...,\alpha_d-p_d-\alpha'_d)$$
and for some $\beta,\beta'\in[-R^n/2,R^n/2]^d$,
$$v_5=(\beta_1-\beta_1',...,\beta_{i+3}+N R^n-\beta_{i+3}',\beta_{i+4}-N R^n-\beta_{i+4}',...,\beta_d-\beta_d').$$
Suppose first that $i=1$. Then
\begin{equation}\label{e.CBBB1}
\left|\frac{(v_4)_4}{(v_5)_4}\right|=\left|\frac{\alpha_4-p_4-\alpha'_4-N R^m}{\beta_{4}+N R^n-\beta_{4}'}\right|\ge \frac{(N-2)R^m}{(N+1) R^n}=\left(\frac{N-2}{N+1}\right)R^{m-n}
\end{equation}
and
\begin{equation}\label{e.CBBB2}
\left|\frac{(v_4)_5}{(v_5)_5}\right|= \left|\frac{\alpha_5-\alpha'_5-p_5}{\beta_5-N R^n-\beta'_5}\right|\le \frac{2 R^m}{(N-1)R^n}=\left(\frac{2}{N-1}\right)R^{m-n}.
\end{equation}
By the choice of $N$, we see that~\eqref{e.CBBB1} and~\eqref{e.CBBB2} are mutually exclusive, so~\eqref{e.disCBBB} follows in the case $i=1$. Suppose instead $2\le i \le d-4$. Then for $R$ large enough,
\begin{equation}\label{e.CBBB3}
\left|\frac{(v_4)_4}{(v_5)_4}\right|=\left|\frac{\alpha_4-p_4-\alpha'_4-NR^m}{\beta_4-\beta_4'}\right|\ge (N-2)R^{m-n}, 
\end{equation}
 and
\begin{equation}\label{e.CBBB4}
\left|\frac{(v_4)_{i+3}}{(v_5)_{i+3}}\right|=\left|\frac{\alpha_{i+3}-p_{i+3}-\alpha_{i+3}'}{\beta_{i+3}+N R^n-\beta_{i+3}'}\right|\le \frac{1.5 R^{m-n}}{N-1}.
\end{equation}
Since~\eqref{e.CBBB3} and~\eqref{e.CBBB4} are mutually exclusive, we get~\eqref{e.disCBBB} also in the case $i\neq 1$.

The statement~\eqref{e.c2dis} follows analogously. 

Next, we show~\eqref{e.disCBCB}. We do this by showing that
$$\CL_{\cc_{1,n}}\cap \CL_{B^1_{R^n}}\mbox{ and }\CL_{\cc_{1,n}}\cap \CL_{B^{d-3}_{R^m}}\mbox{ are disjoint.}$$
Suppose that $L_6\in \CL_{\cc_{1,n}}\cap \CL_{B^1_{R^n}}$ and $L_7\in \CL_{\cc_{1,n}}\cap \CL_{B^{d-3}_{R^m}}$. Let $x\in L_6\cap\cc_{1,n}$, $x'\in L_6\cap B^1_{R^n}$, $y\in L_7\cap \cc_{1,n}$, and $y'\in L_7\cap B^{d-3}_{R^m}$. Let $v_6=x-x'$ and $v_7=y-y'$ be directional vectors of $L_6$ and $L_7$ respectively. Observe that for some $\alpha\in [-R^n/2,R^n/2]^d$ and some $\beta\in  [-R^m/2,R^m/2]^d$ we have
$$v_6=(x_1-p_1-\alpha_1,...,x_4-p_4-\alpha_4-N R^n,...,x_d-p_d-\alpha_d)$$
and
$$v_7=(y_1-p_1-\beta_1,...,y_d-p_d-\beta_d-N R^m).$$
Since $|x_i|,|y_i|\le\sqrt{d}$ for $i=2,...,d$, we get
\begin{equation}\label{e.CBCB1}
\left|\frac{(v_6)_4}{(v_7)_4}\right|=\left|\frac{x_4-p_4-\alpha_4-N R^n}{y_4-p_4-\beta_4}\right|\ge \frac{(N-2)R^n-\sqrt{d}}{\sqrt{d}+R^m}\ge (N-3) R^{n-m},
\end{equation}
provided that $R$ is large enough. We also get 
\begin{equation}\label{e.CBCB2}
\left|\frac{(v_6)_d}{(v_7)_d}\right|=\left|\frac{x_d-p_d-\alpha_d}{y_d-p_d-\beta_d-N R^m}\right|\le \frac{R^n+\sqrt{d}}{(N-2)R^m-\sqrt{d}}\le \frac{R^{n-m}}{N-3},
\end{equation}
when $R$ is large enough. Since~\eqref{e.CBCB1} and~\eqref{e.CBCB2} cannot both hold, we get~\eqref{e.disCBCB}. This completes the proof of the lemma.
\fbox{}\\

In much of what follows, whenever $m$ can be considered fixed, we will simply write $B_1,B_i,\ldots.$
instead of $B_{1,m},B_{i,m},\ldots$ Furthermore, we will say that $f(R)=\Omega(R^{\alpha}),$
if there exists two constants $0<c<C<\infty,$ such that $cR^\alpha\leq f(R) \leq CR^{\alpha}$
for all $R$ large enough.

We will need the following lemma. We formulate it in exactly the way that we will use it, rather than in the
most general way possible. 
\begin{lem}\label{lemma:ball-line}
For any $d\geq 4, m\geq 1$ and boxes $B_1\subset B^1_{R^m},B_{d-3}\subset B^{d-3}_{R^m}$ of sidelength 1,
\begin{equation} \label{eqn:boxcylest}
\BP[\cc_{1,m}\leftrightarrow B_1]=\Omega(R^{-m(d-2)}), \textrm{ and } 
\BP[\cc_{2,m}\leftrightarrow B_{d-3}]=\Omega(R^{-m(d-2)}).
\end{equation}
Furthermore, for any $d\geq 5, i=1,\ldots,d-4$ and pair of boxes $(B_i,B_{i+1})$ of sidelengths 1 such that 
$B_i\subset B^i_{R^m}$ and $B_{i+1}\subset B^{i+1}_{R^m}$
\begin{equation} \label{eqn:boxboxest}
\BP[B_{i}\leftrightarrow B_{i+1}]=\Omega(R^{-m(d-1)}).
\end{equation}
\end{lem}
\noindent
{\bf Remark:}
There are obvious similarities between this lemma and Propositions \ref{prop:dballs}
and \ref{prop:ball-line1}. These propositions will also be used explicitly in the proof.

\medskip

\noindent
{\bf Proof.}
We begin by proving \eqref{eqn:boxboxest}. We note that the distance $d(B_i,B_{i+1})$ 
between the centers of $B_i$ and $B_{i+1}$ can be bounded by 
\begin{eqnarray*}
\lefteqn{d(B_i,q_{i,m})+d(q_{i,m},q_{i+1,m})+d(q_{i+1,m},B_{i+1})}\\
& & \leq \sqrt{d}R^m/2+\sqrt{2}NR^m+\sqrt{d}R^m/2=R^m(\sqrt{2}N+\sqrt{d}).
\end{eqnarray*}
As the boxes contain balls of radius 1, we can use Proposition \ref{prop:dballs} to conclude that 
\[
\mu_{d,1}(\CL_{B_i}\cap \CL_{B_{i+1}})
\geq \frac{c_1}{\left(R^m(\sqrt{2}N+\sqrt{d})\right)^{(d-1)}}
=c_2 R^{-m(d-1)}.
\]
%where $c_1$ is as in Proposition \ref{prop:dballs}. 
Furthermore, since $N=10d+1,$ the constant $c_2$ depends only
on $d$. Using that $1-e^{-x} \geq x/2$ for $x$ small enough, 
%CONTROL: let $f(x)=1-x/2-e^{-x},$ then $f(0)=1$ and note that $f'(x)=e^{-x}-1/2>0$ for $x$ small.
we get that 
\[
\BP[B_{i}\leftrightarrow B_{i+1}]%=1-\BP(B_{i} \not \leftrightarrow B_{i+1})
=1-e^{-u\mu_{d,1}(\CL_{B_i}\cap \CL_{B_{i+1}})}
\geq c_3 R^{-m(d-1)},
\]
for $R$ large enough. Here, $c_3$ depends only on $d$ and $u.$ A similar comment applies to all
numbered constants below.

The distance $d(B_i,B_{i+1})$ can be bounded from below by 
\begin{eqnarray*}
\lefteqn{d(q_{i,m},q_{i+1,m})-d(B_i,q_{i,m})-d(q_{i+1,m},B_{i+1})}\\
& & \geq\sqrt{2}NR^m- \sqrt{d}R^m/2-\sqrt{d}R^m/2=R^m(\sqrt{2}N-\sqrt{d}).
\end{eqnarray*}
Since the boxes $B_i$ and $B_{i+1}$ can be covered by a constant
number of balls of radius 1, we get, using Proposition \ref{prop:dballs}, 
that for $R$ large enough,
\[
\mu_{d,1}(\CL_{B_i}\cap \CL_{B_{i+1}})
\leq \frac{c_4}{\left(R^m(\sqrt{2}N-\sqrt{d})\right)^{(d-1)}}
=c_5 R^{-m(d-1)},
\]
so that 
\[
\BP[B_{i}\leftrightarrow B_{i+1}]
=1-e^{-u\mu_{d,1}(\CL_{B_i}\cap \CL_{B_{i+1}})}
\leq c_5 R^{-m(d-1)}.
\]

We proceed by proving \eqref{eqn:boxcylest} for the event $\{\cc_{1,m}\leftrightarrow B_1\}$. Trivially,
$\BP[\cc_{1,m}\leftrightarrow B_1]\leq \BP[\cc_1\leftrightarrow B_1].$ Furthermore, the distance between 
the center of $B_1$ and the centerline $L_1$ of $\cc_1$ is bounded below by
\[
d(o,q_{1,m})-d(q_{1,m},B_1)\geq NR^m-\sqrt{d}R^m/2=R^m(N-\sqrt{d}/2).
\]
We can therefore use Proposition \ref{prop:ball-line1} to conclude that 
\[
\mu_{d,1}(\CL_{\cc_{1}}\cap \CL_{B_1})\leq \frac{c_6}{\left(R^m(N-\sqrt{d}/2)\right)^{d-2}}
=c_7 R^{-m(d-2)}.
\]
Using that $1-e^{-x}\leq x$ for every $x,$ we get that for $R$ large enough 
\[
\BP[\cc_{1,m}\leftrightarrow B_1] \leq 1-e^{-u\mu_{d,1}(\CL_{\cc_1}\cap \CL_{B_{1}})}
\leq c_7 R^{-m(d-2)}.
\]

In order to establish a lower bound of $\BP[\cc_{1,m}\leftrightarrow B_1],$ we will use a 
similar technique to that of the proof of Proposition \ref{prop:ball-line1}.
To that end, %let $\BZ_m:=\BZ\cap \cc_{1,m}$ and 
consider the collection of balls $\CD_m,$ which is the set of balls $D_i\subset \cc_{1,m}$ of radius $1/8$ 
with center $(i,0\ldots,0)$ for $i\in \BZ.$ Much as in the proof of Proposition \ref{prop:ball-line1}, we note that
\[
\bigcup_{D_i\in\CD_m }{\mathcal L}_{D_i}\subset {\mathcal L}_{\cc_{1,m}},
\] 
so that
\begin{equation}\label{eqn:inclusion}
\bigcup_{D_i\in\CD_m}({\mathcal L}_{D_i}\cap {\mathcal L}_{B_1})
\subset {\mathcal L}_{\cc_{1,m}}\cap {\mathcal L}_{B_1}.
\end{equation}
We will now show that
\begin{equation}\label{eqn:setsdisjoint}
({\mathcal L}_{D_i}\cap {\mathcal L}_{B_1})_{D_i\in\CD_m}
\mbox{ is a disjoint collection of sets of lines.}
\end{equation}
Let $i,j\in \CD_m$ where $i\neq j$ and assume that
\begin{equation}\label{eqn:L1}
L\in {\mathcal L}_{D_i} \cap {\mathcal L}_{D_j}
\end{equation}
and 
\begin{equation}\label{eqn:L2}
L\in {\mathcal L}_{D_i}\cap {\mathcal L}_{B_1}.
\end{equation}
As usual, we write $L$ on the form $L=\{t(k_1,...,k_d)\,:\,-\infty<t<\infty\}+v$
for some $v\in {\mathbb R}^d$. As in the proof of Lemma \ref{lemma:geometric1},
by considering the first and fourth coordinate of the intersections
of $L$ with $D_i,D_j$, we observe that if~\eqref{eqn:L1} holds, then
\begin{equation}\label{eqn:k1}
\left|\frac{k_4}{k_1}\right| \leq \frac{2/8}{|i-j|-2/8}\leq  \frac{2/8}{1-2/8}=\frac{1}{3}.
\end{equation}
Similarly, in order for~\eqref{eqn:L2} to be satisfied, then
\begin{equation}\label{eqn:k2}
\left|\frac{k_4}{k_1}\right| \geq\frac{(N-1)R^m}{R^m}=N-1.
\end{equation}
We conclude that as $N=10d+1,$ ~\eqref{eqn:k1} and ~\eqref{eqn:k2} cannot both hold, which proves ~\eqref{eqn:setsdisjoint}.
Furthermore, we note that for any $D_i\in\CD_m,$
\[
d(B_1,D_i)\leq d(B_1,q_{1,m})+d(q_{1,m},o)+d(o,D_i)\leq \sqrt{d}R^m+NR^m+R^m/2,
\]
where we use that $R\geq 2\max_{i=1,\ldots,d} |p_i|.$ Therefore, by Proposition \ref{prop:dballs}
(and the remark that follows it) , we get that 
\begin{equation} \label{eqn:probdballsest}
\BP(D_i \leftrightarrow B_1)\geq \frac{c_8}{\left((N+2)R^m\right)^{d-1}}.
\end{equation}
Proceeding, we have that
\begin{eqnarray*}
\lefteqn{\mu_{d,1}({\mathcal L}_{{\mathfrak c}(L)}\cap {\mathcal L}_{B_1})
\stackrel{~\eqref{eqn:inclusion}}{\geq}
\mu_{d,1}\left(\bigcup_{D_i\in\CD_m}({\mathcal L}_{D_i}\cap {\mathcal L}_{B_1})\right)}\\
& & \stackrel{~\eqref{eqn:setsdisjoint}}{=}
\sum _{D_i\in\CD_m} \mu_{d,1}({\mathcal L}_{D_i}\cap {\mathcal L}_{B_1})
\stackrel{~\eqref{eqn:probdballsest}}{\geq} |\CD_m|\frac{c_8}{\left((N+2)R^m\right)^{d-1}}
\geq c_9R^{-m(d-2)}.
\end{eqnarray*}
The corresponding statement for the event $\{\cc_{2,m}\leftrightarrow B_{d-3}\}$ follows analogously.
\fbox{}\\

We proceed by proving Lemma \ref{lemma:bound}.
The proof itself contains an elementary geometric claim. The claim is very natural, but
nevertheless requires a proof. In order not to disturb the flow of
the proof proper, we will defer the proof of this claim till later. In what follows, we 
write $f(R)=O(R^\alpha)$ iff there exists a constant $C<\infty,$ such that
$|f(R)| \leq C R^{\alpha}$ for all $R$ large enough. In particular, $O(1)$ refers to a 
function which is bounded for all $R$.\\

\noindent
{\em Proof of Lemma \ref{lemma:bound}.}

Fix $m\geq1.$ We will use the second moment method, i.e. that
\[
\BP[X_{R,m}>0]\geq \frac{\BE[X_{R,m}]^2}{\BE[X_{R,m}^2]},
\]
and proceed by bounding $\BE[X_{R,m}^2].$
Letting $\vec{B'}_m:=(B'_{1,m}\ldots,B'_{d-3,m})$ we have that
\begin{equation}\label{eqn:xr2}
\BE[X_{R,m}^2]=
\BE\left[\sum_{\vec{B}_m}\sum_{\vec{B'}_m}
I(\CP(\vec{B}_m))I(\CP(\vec{B'}_m))
\right].
\end{equation}
For fixed $\vec{B}_m,$ we write $\omega$ as $\omega=\eta\cup\xi,$
where $\eta=\eta_{\vec{B}_m}$ is a Poisson process
of intensity measure $u\mu_{d,1}$ on the set
$\left(\CL_{\cc_{1,m}}\cap \CL_{B_{1,m}}\right) \cup \left(\CL_{B_{1,m}}\cap \CL_{B_{2,m}}\right) \cup
\cdots\cup \left(\CL_{B_{d-3,m}}\cap \CL_{\cc_{2,m}}\right)$ and where $\xi=\xi_{\vec{B}_m}$ is an independent 
Poisson process with the same intensity measure on the complement in $A(d,1).$ Furthermore, from now on
the dependence on $m$ will be dropped from the notation in order to avoid it from being 
overly cumbersome. That is, we will write $\vec{B},B_1,B'_1,\ldots$,  
instead of $\vec{B}_m,B_{1,m},B'_{1,m},\ldots$ However, we will keep the notation $\cc_{i,m}$ as dropping the
emphasis on $m$ changes the meaning.

For any $\eta,$ let
$\eta(\cc_{1,m},B_{1}),\eta(B_{1},B_{2}),\ldots,\eta(B_{d-3},\cc_{2,m})$ denote the restrictions
of $\eta$ onto
$\CL_{\cc_{1,m}}\cap \CL_{B_{1}}, \CL_{B_{1}}\cap \CL_{B_{2}},\ldots, \CL_{B_{d-3}}\cap \CL_{\cc_{2,m}}$
respectively.
We define
\[
S_1(B_{1},\eta):=\{B'_{1}\subset B_{R^m}^1: \CL_{B'_{1}}\cap\eta(\cc_{1,m},B_{1})\neq \emptyset\},
\]
\[
S_2(B_{1},\eta):=\{B'_{1}\subset B_{R^m}^1: \CL_{B'_{1}}\cap\eta(B_{1},B_{2})\neq \emptyset\},
\]
and for $i=2,\ldots,d-4,$
\[
S_1(B_{i},\eta):=\{B'_{i}\subset B_{R^m}^i: \CL_{B'_{i}}\cap\eta(B_{i-1},B_{i})\neq \emptyset\},
\]
\[
S_2(B_{i},\eta):=\{B'_{i}\subset B_{R^m}^i: \CL_{B'_{i}}\cap\eta(B_{i},B_{i+1})\neq \emptyset\},
\]
and finally
\[
S_1(B_{d-3},\eta):=\{B'_{d-3}\subset B_{R^m}^{d-3}: \CL_{B'_{d-3}}\cap\eta(B_{d-4},B_{d-3})\neq \emptyset\},
\]
\[ 
S_2(B_{d-3},\eta):=\{B'_{d-3}\subset B_{R^m}^{d-3}: \CL_{B'_{d-3}}\cap\eta(B_{d-3},\cc_{2,m})\neq \emptyset\}.
\]
Furthermore, for fixed $\vec{B},\vec{B'}$ we define
\[
\CS_{\vec{B}}(\vec{B'},\eta):=\{i\in\{1,\ldots,d-4\}:
B'_i\in S_2(B_i,\eta) \textrm{ and } B'_{i+1}\in S_1(B_{i+1},\eta)\}.
\]
Informally, $i \in \CS_{\vec{B}}(\vec{B'},\eta)$ if there exists $L\in \eta$ connecting
$B'_i$ to $B'_{i+1}.$
Furthermore, we let $\CE_{\vec{B}}(\vec{B'},\eta)$ be the subset of $\{B'_1,B'_{d-3}\}$ which includes
$B_1'$ iff $B'_1\in S_1(B_1,\eta)$ and similarly includes $B_{d-3}'$ iff $B'_{d-3}\in S_2(B_{d-3},\eta).$
%\[
%\CE_{\vec{B}}(\vec{B'}):=\{B'_1,B'_{d-3}: B'_1\in S_1(B_1,\eta) \textrm{} B'_{d-3}\in S_2(B_{d-3},\eta)\}.
%\]
Informally, $\CE_{\vec{B}}(\vec{B'},\eta)$ includes $B'_1$ iff $B'_1$ intersects the lines in $\eta$
connecting $\cc_{1,m}$ to $B_1,$ so that also $B'_1$ is connected to $\cc_{1,m}.$
Given $\vec{B},\vec{B'}$ and $\eta,$ we let
$N_{\vec{B}}(\vec{B'},\eta):=|\CE_{\vec{B}}(\vec{B'},\eta)|+|\CS_{\vec{B}}(\vec{B'},\eta)|$.
We observe that in the special case $d=4,$ a straightforward adjustment of the above definitions 
is needed, since then we only consider one box $B_{R^m}^1.$ We will make no further comment on this.

Noting that the event $\CP(\vec{B})$ is determined by $\eta$ alone,
we get from \eqref{eqn:xr2} that
\begin{eqnarray*}
\lefteqn{\BE[X_{R,m}^2]=\sum_{\vec{B}}\BE\left[\sum_{\vec{B'}}I(\CP(\vec{B}))I(\CP(\vec{B'}))\right]}\\
& & =\sum_{\vec{B}}\BP[\CP(\vec{B})]
\BE\left[\sum_{\vec{B'}}I(\CP(\vec{B'}))\Bigg{|}\eta\in \CP(\vec{B})\right]\\
& & =\sum_{\vec{B}}\BP[\CP(\vec{B})]
\BE\left[\BE\left[\sum_{\vec{B'}}I(\CP(\vec{B'}))\Bigg{|}\eta\right]\Bigg{|}\eta\in \CP(\vec{B})\right]\\
& & =\sum_{\vec{B}}\BP[\CP(\vec{B})]
\BE\left[
\sum_{\vec{B'}:N_{\vec{B}}(\vec{B'},\eta)=0}\BP[\CP(\vec{B'})\mid\eta]
+\sum_{\vec{B'}:N_{\vec{B}}(\vec{B'},\eta)>0}\BP[\CP(\vec{B'})\mid\eta]
\Bigg{|}\eta\in \CP(\vec{B})\right].
\end{eqnarray*}
Note that when we condition on $\eta,$ the only randomness left is in $\xi.$ 

We observe that for $\eta\in \CP(\vec{B})$ and when $N_{\vec{B}}(\vec{B'},\eta)=0,$
\begin{equation}\label{e.nucond}
\BP[\CP(\vec{B'})\mid\eta]\leq \BP[\CP(\vec{B'})].
\end{equation}
The inequality follows since when $N_{\vec{B}}(\vec{B'},\eta)=0,$ there does not exist $L\in \eta$
such that $L\in \left(\CL_{\cc_{1,m}}\cap \CL_{ B'_1}\right)\cup \left(\CL_{B'_1}\cap \CL_{B'_2}\right)\cup\cdots.$
However, it could be that $\eta$ gives partial knowledge of the {\em absence}
of lines in $\omega$ connecting for instance $B'_1$ to $B'_2$. %$\cc_{1,m}$ to $B'_1$. % and similarly $B'_1$ to $B'_2$
%and so on. 
This happens if there exists $L\in \CL_{B_1}\cap \CL_{B_2}$ such that
$L\not \in \eta$ but $L\in \CL_{B'_1}\cap \CL_{B'_2}.$
Continuing, we see that
\begin{eqnarray} \label{eqn:xr2est1}
\lefteqn{\BE[X_{R,m}^2]}\\
& & \stackrel{~\eqref{e.nucond}}{\leq} \sum_{\vec{B}}\BP[\CP(\vec{B})]
\BE\left[
\sum_{\vec{B'}}\BP[\CP(\vec{B'})]
+\sum_{\vec{B'}:N_{\vec{B}}(\vec{B'},\eta)>0}\BP[\CP(\vec{B'})\mid\eta]
\Bigg{|}\eta\in \CP(\vec{B})\right]\nonumber\\
%& & =\sum_{\vec{B}}\BP(I(\CP(\vec{B})))\sum_{\vec{B'}}\BP(I(\CP(\vec{B'})))
%+\sum_{\vec{B}}\BP(I(\CP(\vec{B})))\BE\left[
%\sum_{\vec{B'}:N_{\vec{B}}(\vec{B'},\eta)>0}\BP(I(\CP(\vec{B'}))\mid\eta)
%\Bigg{|}\eta\in \CP(\vec{B}))\right]\\
& & =\BE[X_{R,m}]^2+\sum_{\vec{B}}\BP[\CP(\vec{B})]\BE\left[
\sum_{\vec{B'}:N_{\vec{B}}(\vec{B'},\eta)>0}\BP[\CP(\vec{B'})\mid\eta]
\Bigg{|}\eta\in \CP(\vec{B})\right].\nonumber
\end{eqnarray}
We will proceed by analysing and bounding
\[
\BE\left[\sum_{\vec{B'}:N_{\vec{B}}(\vec{B'},\eta)>0}\BP[\CP(\vec{B'})\mid\eta]
\Bigg{|}\eta\in \CP(\vec{B})\right],
\]
for any fixed $\vec{B}.$

For
%$\vec{k}=(k_1,\ldots,k_{d-1})$ where 
$k_1,\ldots,k_{d-2}\geq 1,$
%and let $\CP(\vec{B},\vec{k})$ be shorthand for the event
%$\{\cc_1\stackrel{k_1}{\leftrightarrow} B_1
%\stackrel{k_2}{\leftrightarrow},\ldots,\stackrel{k_{d-3}}{\leftrightarrow}
%B_{d-3} \stackrel{k_{d-1}}{\leftrightarrow} \cc_2\}.$
using Lemma \ref{lemma:geometric1} and Lemma \ref{lemma:ball-line}, and that the number 
of lines are Poisson distributed,
\begin{eqnarray}\label{eqn:Poiss}
\lefteqn{\BP[\cc_{1,m}\stackrel{k_1}{\leftrightarrow} B_1
\stackrel{k_2}{\leftrightarrow},\ldots,\stackrel{k_{d-3}}{\leftrightarrow}
B_{d-3} \stackrel{k_{d-2}}{\leftrightarrow} \cc_{2,m}
\mid \CP(\vec{B})]}\\
& & =\BP[\cc_{1,m}\stackrel{k_1}{\leftrightarrow} B_1 \mid \cc_1\leftrightarrow B_1 ]
\BP[B_1 \stackrel{k_2}{\leftrightarrow} B_2\mid B_1 \leftrightarrow B_2]\times\cdots
\times \BP[B_{d-3} \stackrel{k_{d-2}}{\leftrightarrow} \cc_{2,m}
\mid B_{d-3} \leftrightarrow \cc_2]\nonumber\\
& & =O(R^{-m(k_1-1)(d-2)})O(R^{-m(k_2-1)(d-1)})\times \cdots \times O(R^{-m(k_{d-2}-1)(d-2)}).\nonumber
\end{eqnarray}
Let 
\[
\CP_{d}(\vec{B}):=\bigcup_{1\leq k_i\leq d \forall i}\{\cc_{1,m}\stackrel{k_1}{\leftrightarrow} B_1
\stackrel{k_2}{\leftrightarrow},\ldots,\stackrel{k_{d-3}}{\leftrightarrow}
B_{d-3} \stackrel{k_{d-2}}{\leftrightarrow} \cc_{2,m}\}.
\]
Using \eqref{eqn:Poiss}, we note that
\[
\BP[\CP_d(\vec{B})^c\mid\CP(\vec{B})]=O(R^{-md(d-2)}).
\]
Therefore, for any positive random variable $Z(\eta)$ bounded above by some finite
number $|Z|,$ we get that
\begin{eqnarray*}
\lefteqn{\BE[Z\mid \eta\in \CP(\vec{B})]}\\
& & =\BE[Z\mid \eta\in \CP_d(\vec{B})]\BP[\CP_d(\vec{B})\mid\CP(\vec{B})]
+\BE[Z\mid \eta\in \CP_d(\vec{B})^c]\BP[\CP_d(\vec{B})^c\mid\CP(\vec{B})]\\
& & \leq \BE[Z\mid \eta\in \CP_d(\vec{B})]+|Z|O(R^{-md(d-2)}).
\end{eqnarray*}
Therefore,
\begin{eqnarray}\label{eqn:xr2est2}
\lefteqn{
\BE\left[\sum_{\vec{B'}:N_{\vec{B}}(\vec{B'},\eta)>0}\BP[\CP(\vec{B'})\mid\eta]
\Bigg{|}\eta\in \CP(\vec{B})\right]}\\
& & \leq \BE\left[\sum_{\vec{B'}:N_{\vec{B}}(\vec{B'},\eta)>0}\BP[\CP(\vec{B'})\mid\eta]
\Bigg{|}\eta\in \CP_d(\vec{B})\right]
+O(R^{-md(d-2)})\sum_{\vec{B'}}1\nonumber\\
& & =O(R^{-md})+\BE\left[\sum_{\vec{B'}:N_{\vec{B}}(\vec{B'},\eta)>0}\BP[\CP(\vec{B'})\mid\eta]
\Bigg{|}\eta\in \CP_d(\vec{B})\right], \nonumber
\end{eqnarray}
since the number of sequences $\vec{B'}$ equals $R^{md(d-3)}.$

Consider now
\begin{equation} \label{eqn:sumeta}
\sum_{\vec{B'}:N_{\vec{B}}(\vec{B'},\eta)>0}\BP[\CP(\vec{B'})\mid\eta]
\end{equation}
for some fixed $\vec{B}$ and $\eta\in \CP_d(\vec{B}).$ As before, after conditioning on $\eta,$
the randomness left is in $\xi.$
In order to get a sufficiently good estimate of \eqref{eqn:sumeta}, we will have to
divide the sum into parts depending on the values of $|\CE_{\vec{B}}(\vec{B'},\eta)|$
and $|\CS_{\vec{B}}(\vec{B'},\eta)|.$ We will then proceed by counting the number
of configurations $\vec{B'}$ corresponding to specific values of $|\CE_{\vec{B}}(\vec{B'},\eta)|$
and $|\CS_{\vec{B}}(\vec{B'},\eta)|,$ and estimate the corresponding probability
$\BP[\CP(\vec{B'})\mid\eta].$ We will start with the latter as that is the easiest part.

Let $l=|\CS_{\vec{B}}(\vec{B'},\eta)|$ and $k=|\CE_{\vec{B}}(\vec{B'},\eta)|.$
Informally, if $k=0,$ then neither of the events $\cc_{1,m} \leftrightarrow B'_1$ and 
$B'_{d-3} \leftrightarrow \cc_{2,m}$ occurs in $\eta.$ If $k=1,$ then exactly one of them occurs in $\eta$, 
while if $k=2,$ both of them occurs in $\eta.$ Similarly, the number of
connections $B'_i \leftrightarrow B'_{i+1}$ that occur in $\eta$ are $l.$ Therefore, in order for
$\CP(\vec{B'})$ to occur, the remaining connections must occur in $\xi.$
Hence, using Lemma \ref{lemma:ball-line}, we get that 
%According to Lemma \ref{lemma:ball-line},
%we have that
%\[
%BP({\mathfrak c}_{1,m}\leftrightarrow B_1)=\Omega\left(\frac{1}{R^{m(d-2)}}\right),
%\ \ \BP(B_{d-3} \leftrightarrow \cc_{2,m})=\Omega\left(\frac{1}{R^{m(d-2)}}\right),
%\]
%and
%\[
%\BP(B_i \leftrightarrow B_{i+1})=\Omega\left( \frac{1}{R^{m(d-1)}}\right) \textrm{ for } i=1,\ldots,d-4,
%\]
%and so we have that
\begin{equation} \label{eqn:probB'bound}
\BP[\CP(\vec{B'})\mid\eta]\leq \Omega(R^{-m(2-k)(d-2)})\Omega(R^{-m(d-4-l)(d-1)}),
\end{equation}
where we made use of Lemma \ref{lemma:geometric1} again. The reason that there is an inequality 
rather than an equality follows much as in \eqref{e.nucond}.
%as $\eta$ might give us negative information of the possibilites of 

In order to bound the number of configurations $\vec{B'}$ such that $l=|\CS_{\vec{B}}(\vec{B'},\eta)|$
and $k=|\CE_{\vec{B}}(\vec{B'},\eta)|,$ we will use the following claim. \\
\medskip

\noindent {\bf Claim:} For $\eta\in\CP_d(\vec{B})$, $j=1,2$ and $i=1,\ldots,d-3$ we have that
$|S_j(B_i,\eta)|=O(R^m).$ Furthermore, for $N=10d+1,$
$|S_1(B_i,\eta)\cap S_2(B_i,\eta)|=O(1)$ for every $i=1,\ldots,d-3.$

\medskip
This claim is very natural. Consider for instance the box $B_1.$
Since $\eta\in \CP_d(\vec{B})$ there are at least one and at most $d$ lines
in both $\eta(\cc_{1,m},B_1)$ and $\eta(B_1,B_2).$ From this, it follows that there can only be
a linear number (in the sidelength of $B^1_{R^m}$) of other boxes $B'_1$ that 
intersects $\eta(\cc_{1,m},B_1)$ {\em or}
$\eta(B_1,B_2).$ Furthermore, due to the positions of the boxes
$B^1_{R^m}$ and $B^2_{R^m},$ the lines in $\eta(\cc_{1,m},B_1)$ will have a large angle
to the lines in $\eta(B_1,B_2).$ Therefore, there cannot be more than some constant number of
boxes $B'_1 \subset B^1_{R^m}$ that intersects the lines in  $\eta(\cc_{1,m},B_1)$
{\em and} $\eta(B_1,B_2).$

\medskip

We will have to consider the different cases $k=0,1,2$ separately. Therefore, assume first that $k=0.$
Recall that we are only considering $N_{\vec{B}}(\vec{B'},\eta)>0$ and thus $l>0$ when $k=0.$
We have that
\begin{eqnarray} \label{eqn:nmbrofbridges0}
\lefteqn{|\{\vec{B'}:|\CE_{\vec{B}}(\vec{B'})|=0, |\CS_{\vec{B}}(\vec{B'})|=l\}|}\\
& & =O(R^{md(d-3)})O(R^{-m(l+1)(d-1)}R^{-m(l-1)})
=O(R^{m(d^2+2-d(l+4))}). \nonumber
\end{eqnarray}
To see this, consider first $l=1$ and assume that only $(B'_1,B'_2)\in\CS_{\vec{B}}(\vec{B'})$.
Then, $B'_1$ and $B'_2$ must be placed along the line(s) in $\eta$ connecting $B_1$ to $B_2.$
By the claim, the number of ways that the pair $(B'_1,B'_2)$ can be chosen decreases from
$O(R^{2md})$ to $O(R^{2m}),$ thus decreasing the total number of ways that $\vec{B'}$
can be chosen by a factor of $O(R^{2m}/R^{2md})=O(R^{-2m(d-1)}).$ 

Consider now $l=2$ and assume again that $(B'_1,B'_2)\in\CS_{\vec{B}}(\vec{B'}).$ If it is the case that
$(B'_3,B'_4)\in\CS_{\vec{B}}(\vec{B'})$ then the total number of ways that the entire sequence
$\vec{B'}$ can be chosen must be of order $O(R^{md(d-3)})O(R^{-4m(d-1)}).$ However, if instead it
is the case that
$(B'_2,B'_3)\in\CS_{\vec{B}}(\vec{B'}),$ then the total number of ways that the entire sequence
$\vec{B'}$ can be chosen must be of order $O(R^{md(d-3)})O(R^{-3m(d-1)})O(R^{-m}).$
The first two factors are explained as above, while the third factor reflects that the box
$B'_2$ must in fact belong to a collection of at most constant size (again using the claim). Continuing in the
same way gives \eqref{eqn:nmbrofbridges0}.

Hence, we conclude, using \eqref{eqn:probB'bound} and \eqref{eqn:nmbrofbridges0}, that
\begin{eqnarray*}
\lefteqn{\sum_{\vec{B'}:|\CE_{\vec{B}}(\vec{B'})|=0, |\CS_{\vec{B}}(\vec{B'})|=l}\BP[\CP(\vec{B'})\mid\eta]}\\
& & = O(R^{m(d^2+2-d(l+4))})O(R^{-2m(d-2)})O(R^{-m(d-4-l)(d-1)})
=O(R^{m(2-l-d)}),
\end{eqnarray*}
so that
\begin{equation} \label{eqn:kzerosuml}
\sum_{l=1}^{d-4}\sum_{\vec{B'}:|\CE_{\vec{B}}(\vec{B'})|=0, |\CS_{\vec{B}}(\vec{B'})|=l}
\BP[\CP(\vec{B'})\mid\eta]
=O(R^{m(1-d)}).
\end{equation}

Consider now $k=1$ and assume without loss of generality that $\CE_{\vec{B}}(\vec{B'})=\{B'_1\}.$ 
The number of sequences $\vec{B'}$
satisfying $l=0,$ is of course $O(R^{md(d-3)})O(R^{-m(d-1)}).$ Furthermore, arguing as in the case
$k=0,$ we see that
%\[
%|\{\vec{B'}:|\CE_{\vec{B}}(\vec{B'})|=1, |\CS_{\vec{B}}(\vec{B'})|=1\}|
%\leq O(R^{d(d-2)})O(R^{-(d-1)})O(R^{-(d-1)})O(R^{-1}),
%\]
%so that
\begin{eqnarray} \label{eqn:nmbrofbridges1}
\lefteqn{|\{\vec{B'}:|\CE_{\vec{B}}(\vec{B'})|=1, |\CS_{\vec{B}}(\vec{B'})|=l\}|}\\
& & =O(R^{md(d-3)})O(R^{-m(l+1)(d-1)}R^{-ml})
=O(R^{m(d^2+1-d(l+4))}). \nonumber
\end{eqnarray}
Therefore, using \eqref{eqn:probB'bound} and \eqref{eqn:nmbrofbridges1},
\begin{eqnarray*}
\lefteqn{\sum_{\vec{B'}:|\CE_{\vec{B}}(\vec{B'})|=1, |\CS_{\vec{B}}(\vec{B'})|=l}\BP[\CP(\vec{B'})\mid\eta]}\\
& & =O(R^{m(d^2+1-d(l+4))})O(R^{-m(d-2)})O(R^{-m(d-4-l)(d-1)})
=O(R^{-m(l+1)}),
\end{eqnarray*}
so that% as in \eqref{eqn:kzerosuml},
\begin{equation} \label{eqn:konesuml}
\sum_{l=0}^{d-4}\sum_{\vec{B'}:|\CE_{\vec{B}}(\vec{B'})|=1, |\CS_{\vec{B}}(\vec{B'})|=l}
\BP[\CP(\vec{B'})\mid\eta]
= O(R^{-m}).
\end{equation}

Finally, we consider $k=2.$ The number of sequences $\vec{B'}$ satisfying $l=0,$ is of course
$O(R^{md(d-3)})O(R^{-2m(d-1)}).$ Furthermore, arguing as in the other cases,
we see that for $l<d-4,$
\begin{eqnarray} \label{eqn:nmbrofbridges2}
\lefteqn{|\{\vec{B'}:|\CE_{\vec{B}}(\vec{B'})|=2, |\CS_{\vec{B}}(\vec{B'})|=l\}|}\\
& & = O(R^{md(d-3)})O(R^{-m(l+2)(d-1)}R^{-ml})
=O(R^{m(d^2+2-d(l+5))}). \nonumber
\end{eqnarray}
Using \eqref{eqn:probB'bound} and \eqref{eqn:nmbrofbridges2}
\begin{eqnarray*}
\lefteqn{\sum_{\vec{B'}:|\CE_{\vec{B}}(\vec{B'})|=2, |\CS_{\vec{B}}(\vec{B'})|=l}\BP[\CP(\vec{B'})\mid\eta]}\\
& & = O(R^{m(d^2+2-d(l+5))})O(R^{-m(d-4-l)(d-1)})
=O(R^{-m(l+2)}).
\end{eqnarray*}
Furthermore,
\begin{eqnarray*}
%\begin{eqnarray} \label{eqn:nmbrofbridges3}
\lefteqn{|\{\vec{B'}:|\CE_{\vec{B}}(\vec{B'})|=2, |\CS_{\vec{B}}(\vec{B'})|=d-4\}|}\\
& & = O(R^{md(d-3)})O(R^{-m(d-3)(d-1)}R^{-m(d-3)})
=O(1), \nonumber
\end{eqnarray*}
so that
\begin{equation} \label{eqn:ktwosuml}
\sum_{l=0}^{d-4}\sum_{\vec{B'}:|\CE_{\vec{B}}(\vec{B'})|=2, |\CS_{\vec{B}}(\vec{B'})|=l}
\BP[\CP(\vec{B'})\mid\eta]
= O(1)+\sum_{l=0}^{d-5}O(R^{-m(l+2)})=O(1).
\end{equation}
Combining \eqref{eqn:kzerosuml},\eqref{eqn:konesuml} and \eqref{eqn:ktwosuml}, we get that
%\eqref{eqn:sumeta} can be estimated by
\begin{equation} \label{eqn:xr2est3}
\sum_{\vec{B'}:N_{\vec{B}}(\vec{B'},\eta)>0}\BP[\CP(\vec{B'})\mid\eta] =O(1).
\end{equation}
Combining \eqref{eqn:xr2est1},\eqref{eqn:xr2est2} and \eqref{eqn:xr2est3} we 
see that there exists a constant $C$ such that for all $R$ large enough,
\[
\BE[X_{R,m}^2]\leq \BE[X_{R,m}]^2+C\BE[X_{R,m}].
\]
Furthermore, by %Proposition \ref{prop:dballs} and 
Lemma \ref{lemma:ball-line} there exists a constant $C'>0$ such that for every $\vec{B},$
$\BP[\CP(\vec{B})]\geq C'R^{-2m(d-2)}R^{-m(d-4)(d-1)}$.
Therefore,
\[
\BE[X_{R,m}]\geq R^{md(d-3)}C'R^{-2m(d-2)}R^{-m(d-4)(d-1)}=C',
\]
so that
\[
\frac{\BE[X_{R,m}]^2}{\BE[X_{R,m}^2]}
\geq \frac{\BE[X_{R,m}]^2}{\BE[X_{R,m}]^2+C\BE[X_{R,m}]}
=\frac{1}{1+C\BE[X_{R,m}]^{-1}}
\geq \frac{1}{1+C/C'}>0,
\]
for all $R$ large enough. This proves the statement.
\fbox{}\\

\noindent
We will now prove the claim.\\

\noindent
{\bf Proof of Claim.} Recall that $\eta\in \CP_d(\vec{B}),$ so that
the first part, i.e. $|S_j(B_i,\eta)|=O(R^m)$ follows from the fact that the number of
lines in $\eta(\cc_{1,m},B_1),\eta(B_1,B_2),\ldots$ are bounded by $d.$

Let $R$ be so large that $R/2>\max_{i=3,\ldots,d}|p_i|+1.$
Consider first any pair of lines $L_1,L_2$
such that $L_1\in \CL_{\cc_{1,m}}\cap \CL_{B_1}$ and $L_2\in \CL_{B_1}\cap \CL_{B_2}$
where $B_1\subset B_{R^m}^1$ and $B_2\subset B_{R^m}^2.$ Let $x\in L_1\cap \cc_{1,m}$ and
$x'\in L_1\cap B_1,$ and note that $v_1=x-x'$ is a directional vector for $L_1.$ In the same way,
letting $y\in L_2\cap B_1$ and $y'\in L_2\cap B_2,$ $v_2=y-y'$ becomes a directional vector 
for $L_2.$
Furthermore, we can write
$x'=q_{1,m}+\alpha$ where $\alpha\in[-R^m/2,R^m/2]^{d},$
$y=q_{1,m}+\alpha+\gamma$ where $\gamma \in [-1,1]^d$
and $y'=q_{2,m}+\beta$ where $\beta\in[-R^m/2,R^m/2]^{d}.$

Considering the angle $\theta$ between the lines $L_1$ and $L_2,$ we have that
\begin{equation} \label{eqn:angle1}
\cos \theta =\frac{\langle v_1,v_2\rangle}{|v_1| |v_2|}.
\end{equation}
By showing that $|\cos \theta|$ is uniformly bounded away from 1 when $N=10d+1$ and $R$ is large,
the second part of the claim is established.
We note that
\[
|v_1|^2=|p+NR^me_{4}+\alpha-x|
 \geq |\alpha_4+p_4+NR^m-x_4|^2
\geq (N-1)^2R^{2m},
\]
%\begin{eqnarray*}
%\lefteqn{|v_1|^2=|q_{1,m}+\alpha-x|^2=|p+NR^me_{4}+\alpha-x|}\\
%& & \geq |\alpha_1+p_1-x_1|^2+|\alpha_4+p_4+NR^m-x_4|^2
%\geq |\alpha_1+p_1-x_1|^2+(N-1)^2R^{2m},
%\end{eqnarray*}
and similarly,
\[
|v|^2=|q_{2,m}+\beta-q_{1,m}-\alpha-\gamma|^2\geq |\beta_4-NR^m-\alpha_4-\gamma_4|^2
+|\beta_5+NR^m-\alpha_5-\gamma_5|^2
\geq 2((N-1)R^m-1)^2.
\]
Furthermore,
\begin{eqnarray*}
\lefteqn{\langle v_1,v_2\rangle=\langle q_{1,m}+\alpha-x,q_{2,m}+\beta-q_{1,m}-\alpha-\gamma \rangle}\\
& & =\sum_{i\neq 4,5}(p_i+\alpha_i-x_i)(\beta_i-\alpha_i-\gamma_i)
+(p_4+NR^m+\alpha_4-x_4)(\beta_4-\alpha_4-NR^m-\gamma_4)\\
& & \ \ \ \ +(p_5+\alpha_5-x_5)(\beta_5+NR^m-\alpha_5-\gamma_5),
\end{eqnarray*}
so that using $|x_1|\leq R^m/2$ and $|x_i|\leq 1$ for $i\neq 1,$
\begin{eqnarray*}
\lefteqn{|\langle v_1,v_2\rangle|}\\
%& & \leq (R^m+|x_1|)(R^m+1)+(d-3)(R^m+1)^2+((N+1)R^m+1)^2+(R^m+1)((N+1)R^m+1)\\
& & \leq 2R^m(R^m+1)+(d-3)(R^m+1)^2+((N+1)R^m+1)^2+(R^m+1)((N+1)R^m+1) \\
%& & =R^{2m}(2+d-3+(N+1)^2+(N+1))+O(R)\\
& & =R^{2m}(d+(N+1)^2+N)+O(R^m).
%\\
%& & =O(R^2)+O(|\delta_1 |R)+O(N^2R^2)+O(NR^2)+O(NR).
\end{eqnarray*}
Therefore, we get from \eqref{eqn:angle1} and that $N=10d+1,$ 
\begin{eqnarray} \label{eqn:angle2}
\lefteqn{|\cos \theta|\leq
\frac{R^{2m}(d+(N+1)^2+N)+O(R^{m})}
{(N-1)R^{m}\sqrt{2}((N-1)R^m-1)}}\\
& & =\frac{(d+(10d+2)^2+10d+1)+O(R^{-m})}
{10d\sqrt{2}(10d-O(R^{-m}))} \nonumber\\
& & =\frac{(100d^2+51d+5)+O(R^{-m})}
{100d^2\sqrt{2}} \left(1+O(R^{-m})\right).\nonumber
\end{eqnarray}
Since $(100d^2+51d+5)/100d^2$ is decreasing in $d,$ we can estimate the RHS by inserting $d=4$
in which case we get the bound
\begin{equation} \label{eqn:angleest1}
\frac{1809}{1600\sqrt{2}}\left(1+O(R^{-m})\right)\leq \frac{1.14}{\sqrt{2}}\left(1+O(R^{-m})\right),
\end{equation}
which is uniformly bounded away from 1 for $R$ large enough.
Therefore, the angle between $L_1$ and $L_2$ must be uniformly (in $d$ and in the choice of
$L_1,L_2$) bounded away from 0 for all $R$ large enough. From this, the claim follows for all such
$L_1$ and $L_2.$

The remaining cases (i.e. when $L_2\in \CL_{B_1}\cap \CL_{B_2}$ and $L_3\in \CL_{B_2}\cap \CL_{B_3}$ etc)
are handled in the same way. The final case when
$L_{d-3}\in \CL_{B_{d-4}}\cap \CL_{B_{d-3}}$ and $L_{d-2}\in \CL_{B_{d-3}}\cap \cc_{2,m}$ is
somewhat {\em more} technical than the current case due to the position of the cylinder
$\cc_2.$ However, the approach is completely analogous.
\fbox{}\\

%{\bf Acknowledgments:}

\end{document}